\documentclass[a4 paper]{article}

\usepackage{fancyhdr}
\usepackage{latexsym}
\usepackage{mathrsfs}
\usepackage{cancel}
\usepackage{color}
\usepackage{multirow}
\usepackage{eso-pic}
\usepackage{dsfont}
\usepackage{amssymb}
\usepackage{graphicx}
\usepackage{setspace}
\usepackage{geometry}
\usepackage{amsthm}
\usepackage{hyperref}
\usepackage[intlimits]{amsmath}
\usepackage{hyperref}
\usepackage{tikz-cd}
\usepackage{amssymb,amsfonts}
\usepackage[all,arc]{xy}
\usepackage{enumerate}
\usepackage{mathrsfs}
\usepackage{tikz}
\usepackage{bbm}
\usepackage{graphicx}
\newcommand{\wt}{\widetilde}
\newcommand{\wh}{\widehat}

\newcommand{\End}{\mathrm{End}}

\newcommand{\Hom}{\mathrm{Hom}}

\newcommand{\Mod}{\mathrm{-Mod}}
\newcommand{\Proj}{\mathrm{Proj}}

\newcommand{\lp}{\left(}
\newcommand{\rp}{\right)}

\newcommand{\Ind}{\mathrm{Ind}}

\newcommand{\Coh}{\mathrm{Coh}}
\newcommand{\QCoh}{\mathrm{QCoh}}
\newcommand{\IndCoh}{\mathrm{IndCoh}}
\newcommand{\Perf}{\mathrm{Perf}}

\newcommand{\lbb}{[\![}
\newcommand{\rbb}{]\!]}
\newcommand{\HR}{\CM_{G, T^*V}}
\newcommand{\sHR}{[\CM_{G,T^*V}]}
\newcommand{\match}{\rhd\!\!\!\lhd }

\newcommand{\hypquotient}{/\!\!/\!\!/}

\newcommand{\C}{\mathbb C}

\newcommand{\Z}{\mathbb Z}

\newcommand{\fg}{\mathfrak{g}}
\newcommand{\fh}{\mathfrak{h}}

\newcommand{\tgv}{\mathfrak{d}_{\mathfrak g, V}}
\newcommand{\fa}{\mathfrak{a}}

\newcommand{\CA}{{\mathcal A}}
\newcommand{\CB}{{\mathcal B}}
\newcommand{\CC}{{\mathcal C}}

\newcommand{\CF}{{\mathcal F}}
\newcommand{\CG}{{\mathcal G}}
\newcommand{\CH}{{\mathcal H}}

\newcommand{\CL}{{\mathcal L}}
\newcommand{\CM}{{\mathcal M}}
\newcommand{\CN}{{\mathcal N}}
\newcommand{\CO}{{\mathcal O}}

\newcommand{\CT}{{\mathcal T}}

\newcommand{\CX}{{\mathcal X}}

\newcommand{\be}{\begin{equation}}
\newcommand{\ee}{\end{equation}}

\newcommand{\btik}{\begin{tikzcd}}
\newcommand{\etik}{\end{tikzcd}}

\begin{document}

\title{Quantum Groups and Symplectic Reductions}

\author{Wenjun Niu}

\newtheorem{Def}{Definition}[section]
\newtheorem{Thm}[Def]{Theorem}
\newtheorem{Prop}[Def]{Proposition}
\newtheorem{Cor}[Def]{Corollary}
\newtheorem{Lem}[Def]{Lemma}
\newtheorem{Rem}[Def]{Remark}

\numberwithin{equation}{section}

\maketitle

\abstract{Let $G$ be a reductive algebraic group with Lie algebra $\fg$ and $V$ a finite-dimensional representation of $G$. Costello-Gaiotto studied a graded Lie algebra $\tgv$ and the associated affine Kac-Moody algebra. In this paper, we show that this Lie algebra can be made into a sheaf of Lie algebras over $T^*[V/G]=[\mu^{-1}(0)/G]$, where $\mu: T^*V\to \fg^*$ is the moment map. We identify this sheaf of Lie algebras with the tangent Lie algebra of the stack $T^*[V/G]$. Moreover, we show that there is an equivalence of braided tensor categories between the bounded derived category of graded modules of $\tgv$ and graded perfect complexes of $[\mu^{-1}(0)/G]$.}

\tableofcontents

\newpage

\section{Introduction}

Let $G$ be a reductive algebraic group with Lie algebra $\fg$ and $V$ a finite-dimensional representation. In \cite{costello2019vertex}, the authors considered a graded Lie algebra, to be denoted by $\tgv$, whose affine Kac-Moody vertex algebra appears as the (perturbative) \textit{boundary vertex algebra} of the B twist of the 3d $\CN=4$  Langrangian gauge theory determined by $(G, T^*V)$. The Lie algebra $\tgv$ is simply given by $T^*[-2]\fh$ where $\fh=\fg\ltimes V[-1]$. 

Let $\mu: T^*V\to \fg^*$, and denote by $\CM:=\mu^{-1}(0)/\!/G$ the symplectic reduction (as an affine variety) and $[\CM]:=[\mu^{-1}(0)/G]$ the quotient stack (when $\mu: T^*V\to \fg^*$ is not flat, the variety $\mu^{-1}(0)$ is taken to be the derived zero locus). The variety $\CM$ and the stack $[\CM]$ are interesting objects in the area of symplectic geometry as $\CM$ is a Poisson cone and in good cases has conical symplectic singularities, whose resolution can be constructed from open substack of $[\CM]$. These not only give a large zoo of symplectic varieties, but they also provide ample examples of symplectic duality \cite{braden2016conical, braden2016conical2, webster20233}. From the perspective of the aforementioned field theory, the variety $\CM$ is known to be the \textit{Higgs branch}, and symplectic duality is a mathematical manifestation of 3d mirror symmetry of \cite{intriligator1996mirror}. In particular, the variety $\CM$ is expected (and verified in many examples) to be the Coulomb branch of the mirror dual theory. In the case that such a theory is again a gauge theory, such a Coulomb branch is defined in the influential works \cite{nakajima2016towards, braverman2018towards}.

The goal of this paper is two-fold. First, we bring the Lie algebra $\tgv$ into geometry by showing that it can be made into a  sheaf of Lie algebra over the stack $[\CM]$, which we denote by $l_{[\CM]}$, whose restriction to the cone point $0\in [\CM]$ recovers $\tgv$. More specifically, the object:
\be
l_{[\CM]}:=\C[\mu^{-1}(0)]\otimes \tgv, \qquad d=\sum x_i\otimes[\psi_+^i, -]+y_i\otimes [\psi_-^i, -]
\ee
with its canonical $G$-equivariant structure and Lie algebra structure from $\tgv$ becomes a Lie algebra object in $\Coh ([\CM])$. In fact, we show that this sheaf of Lie algebra is simply the tangent Lie algebra of $[\CM]$. The main statement is the following:

\begin{Prop}[Proposition \ref{Prop:TM}]

The tangent Lie algebra $T_{[\CM]}[-1]$ is identified with $l_{[\CM]}$ as a Lie algebra object in $\Coh ([\CM])$. Under this identification, the quadratic form on $\tgv$ is identified with the symplectic form on $T_{[\CM]}[-1]$. 

\end{Prop}

Second, we bring geometry to algebra by showing that the category of sheaves on $[\CM]$ is controlled by modules of the algebra $\tgv$. More specifically, we denote by $U(\tgv)\lbb\hbar\rbb\Mod_G^{gr}$ the category of graded modules of $U(\tgv)\lbb\hbar\rbb$ finitely generated and flat over $\C\lbb\hbar\rbb$, whose action of $\fg$ integrates to an algebraic action of $G$. Here the grading on $\hbar$ is $-2$. Applying the works of  \cite{drinfeld1986quantum, drinfeld1991quasi, etingof1996quantization}, this category can be given two equivalent braided tensor structures, by giving $U(\tgv)\lbb\hbar\rbb$ the structure of a quasi-triangular quasi-Hopf algebra as well as a quasi-triangular Hopf algebra. On the other hand, let $\Perf_{\C^\times}([\CM])$ be the category of perfect complexes on $[\CM]$. This category can be endowed with a braided tensor structure similar to \cite{roberts2010rozansky}, based on the study of Rozansky-Witten theory \cite{kapustin2009three, kapustin2010three}. The identification of tangent Lie algebras will be upgraded to the following statement, whose proof is a simple repeat of the Koszul duality of \cite{beilinson1996koszul}. 

\begin{Thm}[Theorem \ref{Thm:tangentBTC}]\label{Thm:tangentBTCintro}

There exists an equivalence of braided tensor categories:
\be
\CF: D^b U(\tgv)\lbb\hbar\rbb\Mod_G^{gr}\simeq \Perf_{\C^\times}([\CM]\lbb\hbar\rbb)
\ee
such that the Lie algebra object $l_{[\CM]}$ is precisely the image of $\tgv$ (as the adjoint module) under the functor $\CF$. 

\end{Thm}

\begin{Rem}

This Koszul duality statement holds very generally. For example, in \cite{baranovsky2005bgg, baranovsky2007bgg} it was proven for $L_\infty$ algebras arising as the tangent Lie algebra of complete intersections of projective and toric varieties. 

\end{Rem}

\subsection{Braided tensor categories and TQFT}

\subsubsection{B twist of 3d $\CN=4$ theories and quantum groups}

The B twist of 3d $\CN=4$ gauge theory is a (physical) \textit{topological quantum field theory} in 3 dimensions. As such, it gives rise to a braided tensor category $\CC$, the category of line operators. Understanding of this category is very important to the study of the TQFT, since one can hope to construct knot invariants and 3-manifold invariants using the method of \cite{reshetikhin1991invariants}. Moreover, the algebra $\C[\CM]$ (the algebra of local operators) can be computed from $\CC$ as the endomorphism of the identity object:
\be\label{eq:endvac}
\End_{\CC} (\mathbbm{1})\cong \C[\CM]. 
\ee
We must note here that one important distinction between topological twists of 3d $\CN=4$ theories and the Chern-Simons theories considered in \cite{reshetikhin1991invariants}, is that the category $\CC$ is generally not a semi-simple category, but rather the derived category of a non-semisimple tensor category. The endomorphism in equation \eqref{eq:endvac} is taken to be the derived endomorphism. Non-semisimplicity presents immense difficulty in the construction of 3-manifold invariants.

For a restricted class of non-semi-simple braided tensor categories, the work \cite{costantino2014quantum} extended the construction of  \cite{reshetikhin1991invariants} to obtain 3-manifold invariants. Their primary examples are module categories of quantum groups at roots of unity. See also \cite{de20223} for further information and references. Although the braided tensor categories for 3d $\CN=4$ gauge theories usually don't fall into the class considered in \cite{costantino2014quantum}, it is hopeful that their work could be extended to the 3d $\CN=4$ gauge theory settings.  Nevertheless, one needs to understand the category $\CC$ as much as possible, including all the objects, the braiding between them, their rigid dual and so on. Otherwise obtaining 3-manifold invariants would be hopeless. 

So far, there has been a large body of work going into the study of the category of line operators of the topological twists of 3d $\CN=4$ gauge theories, including but not limited to \cite{bullimore2016boundaries, braverman2019coulomb3, webster2019coherent, oblomkov2024categorical, dimofte2020mirror, costello2019higgs, creutzig2021qft, hilburn2022tate, ballin20223d, BCDN, gammage2022betti, Garner:2023vb, geer2022three, garner2024b}. However, the understanding of the braided-tensor structure of this category is still lacking. In recent work \cite{creutzig2024kazhdan} and work to appear \cite{cdntoappear}, this has been significantly impoved, by arguing that the category admits a quantum group description. Namely, it is argued that the category $\CC$ is given by the derived category of modules of an explicit (topological) Hopf algebra $\CH$ with an $R$ matrix in some algebraic completion of $\CH \otimes \CH$. We show in  \cite{cdntoappear} that the quantum group considered in this paper $U(\tgv)\lbb\hbar\rbb$ shows up as a Hopf sub-algebra of $\CH$. 

Viewed from this vein, Theorem \ref{Thm:tangentBTCintro} further confirms the results of \cite{creutzig2024kazhdan, cdntoappear}, by directly connecting modules of $U(\tgv)\lbb\hbar\rbb$ with sheaves on the Higgs branch $\CM$ (or better, the stack $[\CM]$). Moreover, it provides a clear geometric meaning to the Lie algebra $\tgv$ and $\CH$, in terms of the tangent Lie algebra of $[\CM]$ and its quantization. 

\subsubsection{Rozansky-Witten TQFT on cotangent stacks}

The B twist of 3d $\CN=4$ gauge theories are closely related to another class of TQFT named Rozansky-Witten theory \cite{rozansky1997hyper}. In fact, the study of line operators in 3d $\CN=4$ gauge theories took much inspiration from the works devoted to the study of Rozansky-Witten TQFT \cite{kontsevich1999rozansky, kapranov1999rozansky, kapustin2009three, Kapustin:2009cd, kapustin2010three, roberts2010rozansky}. In particular, given a smooth symplectic variety $X$, the category of line operators for the Rozansky-Witten theory with target $X$ is proposed to be $\Coh (X)$, the category of coherent sheaves on $X$, such that the braided tensor structure is induced from the tangent Lie algebra $T_X[-1]$ and the Poisson form (see \ref{subsec:symquotangent} for more details). 

The B twist of 3d $\CN=4$ gauge theory defined by $(G, T^*V)$ can be viewed as a Rozansky-Witten TQFT valued in the co-tangent stack $T^*(V/G)=[\CM]$. Thanks to developments in derived geometry  (\cite{pantev2013shifted, safronov2016quasi, hennion2018tangent, gaitsgory2019study, gaitsgory2017study} to name a few), the theory of tangent Lie algebras and symplectic structures are extended to stacks like $[\CM]$. In particular, the category $\Coh ([\CM])$ and $\Perf([\CM])$ has a braided tensor category structure similar to the smooth case. 

Viewed from this vein, Theorem \ref{Thm:tangentBTCintro} confirms the relation between Rozansky-Witten TQFT and 3d $\CN=4$ gauge theories. Moreover, it simplifies the study of Rozansky-Witten TQFT in the following two ways. Firstly, recognizing the relation between the tangent Lie algebra of $[\CM]$ and $\tgv$ significantly improves our undersanding of the braided tensor structure of $\Coh([\CM])$ (for example, one can write down explicit maps between chain complexes that computes braiding). Secondly, it shows that the category $\Coh([\CM])$ admits a fiber functor, which is exact for the Koszul t-structure (by transporting the fiber functor of \cite{etingof1996quantization}). This fiber functor is roughly speaking given by the following:
\be
\Hom \lp \CO_{[V/G]}\otimes \CO_{T^*_0}, -\rp
\ee
where $\CO_{[V/G]}$ is the structure sheaf of the base $V/G$ of $T^*([V/G])$ and $\CO_{T^*_0}$ is the structure sheaf of the fiber at $0$. These two sheaves (in fact, $\mathbb{E}_2$ algebra objects in the braided tensor category) corresponds to the transverse boundary conditions considered in \cite{cdntoappear}. Note that in the derived category $\CO_{[V/G]}\otimes \CO_{T^*_0}\cong \CO_{[e]}$, where $\CO_{[e]}$ is the sky-scraper sheaf at the cone point. A nice realization from the construction of the equivalence is that one can show that this is a fiber functor (namely that it is monoidal) only if one keeps the explicit chain model of $\CO_{[e]}\cong \CO_{[V/G]}\otimes \CO_{T^*_0}$ transported from $\CF$.

\subsection{Boundary VOA and Kazhdan-Lusztig equivalences}

As mentioned earlier, the affine Kac-Moody algebra $V(\tgv)$ of $\tgv$ shows up as perturbative boundary vertex algebras of B twist of 3d $\CN=4$ gauge theories \cite{costello2019vertex}. Here the bi-linear form is taken to be $\kappa+k\kappa_\fg$, where $\kappa$ is the pairing between $\fh$ and its dual and $\kappa_\fg$ is the non-degenerate pairing on $\fg$. The value of $k$ is determined in \cite[Section 6.1]{costello2020boundary}. The work \cite{ballin20223d, BCDN, creutzig2024kazhdan} used certain non-perturbative completion of VOA for the case when $G=(\C^\times)^r$ is abelian. However, when studying the category of modules, it is still most useful to start with the Kazhdan-Lusztig category of $V(\tgv)$. Unfortunately, for non-abelian $G$, this category has not been studied very much. 

In the case when $G$ is abelian, it was also shown in \cite{creutzig2024kazhdan} that there is a Kazhdan-Lusztig type equivalence, namely an equivalence between modules of $V(\tgv)$ and modules of a quasi-triangular Hopf algebra $U_q (\tgv)$ built from $\tgv$. We won't comment it here, but the quantum group $U_q (\tgv)$ is essentially a finite-evaluation of the quasi-triangular Hopf algebra $U(\tgv)\lbb\hbar\rbb$ (evaluating $e^\hbar$ at $q$).

It is not entirely suprising that such a Kazhdan-Lusztig type correspondence should exist, since the B twist of 3d $\CN=4$ gauge theories can be viewed as a supergroup Chern-Simons theory \cite{Kapustin:2009cd}, whose Lie superalgebra is precisely $\tgv$. 

From Theorem \ref{Thm:tangentBTCintro}, we therefore expect a Kazhdan-Lusztig equivalence for $V(\tgv)$ beyond abelian case. More specifically, we expect that for any reductive algebraic group $G$, certain full subcategory of $V(\tgv)$ modules is equivalent to modules of a finite-evaluation of $U(\tgv)\lbb\hbar\rbb$ (which we conjecture exists). We also expect that a careful non-perturbative completion of $V(\tgv)$ similar to the manner of \cite{BCDN} gives rise to a braided tensor category equivalent to the category of modules of the quantum group $\CH$ from \cite{cdntoappear}. We expect that this equivalence can be realized by considering the same fiber functor of \cite{etingof1996quantization}, replacing the Verma modules by free-field realizations of $V(\tgv)$, as in \cite{creutzig2024kazhdan} for the case of abelian gauge group $G$. 

\subsection{Organization and sketch of proof strategies}

In Section \ref{sec:Takiffquant}, we introduce the graded Lie algebra $\tgv$ and its two quantizations according to \cite{drinfeld1986quantum, drinfeld1991quasi}. Namely, we will endow the algebra $U(\tgv)\lbb\hbar\rbb$ with the structure of a quasi-triangular quasi-Hopf algebra and a quasi-triangular Hopf algebra. We then recall the work of \cite{etingof1996quantization}, from which we derive the equivalence of these two quantizations, namely the equivalence of braided tensor categories defined by these two structures, on the category of modules finitely-generated and flat over $\C\lbb\hbar\rbb$.

In Section \ref{sec: RWtangent}, we study the symplectic reduction $T^*([V/G])$. We recall the theory of Tangent Lie algebras for stacks according to \cite{hennion2018tangent, gaitsgory2019study, gaitsgory2017study}, extending the idea of \cite{kapranov1999rozansky}. This, together with the results of derived symplectic geometry \cite{pantev2013shifted, safronov2016quasi}, allows us to define a braided tensor category structure on the category of coherent sheaves and perfect complexes on $T^*([V/G])$, extending the smooth case of \cite{roberts2010rozansky}. 

In Section \ref{sec:SRandTakiff}, we first use the definition of \cite{gaitsgory2019study, gaitsgory2017study} to explicitly calculate the tangent Lie algebra of $T^*([V/G])$. We in particular show that the Lie algebra $\tgv$ can be naturally made into a sheaf of Lie algebras over $T^*[V/G]$, and is naturally identified with the tangent Lie algebra. We then proceed to prove that one can upgrade this to an equivalence $\CF$ of braided tensor categories. The proof of this equivalence at the derived category level follows word-to-word from the proof of the Koszul duality statement of \cite{beilinson1996koszul}, by introducing a larger weighted category so that the inverse of $\CF$ is defined on the chain level. We then identify the action of a module of $\tgv$ under $\CF$ with the action given by Atiyah class, from which the equivalence of braided tensor categories follow immediately.

\begin{Rem}

In this paper, we only concern ourselves with the classical triangulated categories, which in most cases are derived categories of some abelian categories. The theory of tangent Lie algebra for stacks is built on the powerful machinery of derived geometry and DG categories, from which we can pull-back results to these classical categories. We expect that most of the statements in this paper can be simply proven there as well, but we prefer to stick to the more down-to-earth setting, leaving the task of such extensions to more capable hands.  

\end{Rem}

\subsection{Acknowledgements}

I would like to thank  Thomas Creutzig and Tudor Dimofte for collaborations on related topics that led to this work. I would like to thank Kevin Costello, Davide Gaiotto, Niklas Garner, Eugene Gorsky, Benjamin Hennion, Justin Hilburn, Lev Rozansky, Nick Rozenblyum, Ben Webster and Matthew Young for very helpful discussions related to this project. I would like to thank my best friend Don Manuel for his encouragements. My research is supported by Perimeter Institute for Theoretical Physics. Research at Perimeter Institute is supported in part by the Government of Canada through the Department of Innovation, Science and Economic Development Canada and by the Province of Ontario through the Ministry of Colleges and Universities.

\section{Lie algebra $\tgv$ and its quantizations}\label{sec:Takiffquant}

\subsection{Definition of $\tgv$ as a double}\label{subsec:Takiff}

In the following, we will always denote by $G$ a reductive algebraic group with $\mathrm{Lie}(G)=\fg$, and $V$ a finite dimensional representation of $G$. From this, one obtain a graded Lie algebra $\fh:=\fg\ltimes V[-1]$. From now on, we always denote by $[-,-]$ the graded commutator, so that if $x, y$ are two homogenous elements in a graded algebra, then:
\be
[x, y]=xy-(-1)^{|x||y|}yx.
\ee
 If we fix a set of basis $\{x_a\}$ of $\fg$, and $\{\psi_+^i\}\in V$, then the nontrivial commutation relations of $\fh$ are given by:
\be
[x_a, x_b]=\sum_c f_{ab}^c x_c,\qquad [x_a, \psi_+^i]=\rho(x_a)^i{}_j \psi_+^j=x_a\cdot \psi_+^i,
\ee
where $f_{ab}^c$ is the structure constants of $\fg$, and $\rho$ is the matrix elements of the action of $\fg$ on $V$. One can always view $\fh$ as a Lie bi-algebra with trivial Lie co-bracket. This gives rise to the double $D(\fh)$, which as a vector space is $\fh\oplus \fh^*$. 

\begin{Def}\label{def:grTakiff}
We define the Lie algebra $\tgv$ to be the double $D(\fh)$. This can be made into a graded Lie algebra, where the degree of $\fh^*$ is shifted by $2$:
\be
\tgv=\fh\ltimes \fh^*[-2]
\ee 

\end{Def}  

\begin{Rem}
As discussed in the introduction, this Lie algebra is also equal to $T^*[-2]\fh$. 

\end{Rem}

The Lie algebra $\tgv$ is a solvable Lie algebra which, as a graded vector space, can be represented by:
\be
\tgv=\fg\oplus (V\oplus V^*)[-1]\oplus \fg^*[-2],
\ee
If we fix a dual basis $\{t^a\}$ of $\fg^*[-2]$ and $\{\psi_{-,i}\}$ of $V^*$, then the additional nontrivial commutation relations of $\tgv$ are given by:
\be
[x_a, t^b]=\sum -f_{ac}^b t^c,\qquad [x_a, \psi_{-,i}]=-\rho(x_a)^j{}_i \psi_{-,j},\qquad [\psi_+^i, \psi_{-,j}]=\sum_a \rho^i{}_j(x_a)t^a.
\ee
In particular $[\fh^*, \fh^*]=0$. If we denote by $\tgv^+$ the Lie subalgebra generated by elements in positive degrees, then it is easy to see that $\tgv=\fg\ltimes \tgv^+$. 

As a double, $\tgv$ came equipped with a non-degenerate invariant bilinear form $\kappa$ (of degree $2$), which is defined via the dual pairing between $\fh$ and $\fh^*$. This gives rise to the following classical $r$ matrix:
\be\label{eq:classicalr}
r=\sum_a x_a\otimes t^a +\sum_i \psi_+^i\otimes \psi_{-,i}\in \fh\otimes \fh^*,
\ee
This element $r$ satisfies:
\be\label{eq:classicalC}
\Omega:=\frac{1}{2}(r+r^{21})=\frac{1}{2}\lp \Delta (C)-C\otimes 1-1\otimes C\rp,
\ee
where $C$ is the quadratic Casimir associated to $\kappa$, namely:
\be
C=\frac{1}{2}\sum_a \lp x_a t^a+t^a x_a\rp +\frac{1}{2}\sum_i (\psi_+^i \psi_{-,i}- \psi_{-,i} \psi_+^i)\in U(\tgv),
\ee
and $\Delta$ is the standard co-product on $U(\tgv)$. This $r$ satisfies the classical \textit{Yang-Baxter} equation:
\be
[r^{12}, r^{13}]+[r^{12}, r^{23}]+[r^{13}, r^{23}]=0\in U(\tgv)^{\otimes 3}.
\ee
Here $r^{ij}$ means putting $r$ in the $i,j$-th components of the tensor product. Consequently, $r$ gives rise to a Lie bi-algebra structure on $\tgv$ such that:
\be
\delta_r (x)=[x\otimes 1+1\otimes x, r].
\ee
Explicitly, we have:
\be
\delta_r(x_a)=\delta_r(\psi_+^i)=0,\qquad \delta_r(t^a)=-\sum f_{bc}^a t^b\otimes t^c,\qquad \delta_r(\psi_-^i)=\sum \rho(x_a)^{i}{}_j (\psi_-^j\otimes t^a+t^a\otimes\psi_-^j). 
\ee
If we simply choose a set of graded basis $\{x_i\}\subseteq \fh$ and corresponding set of graded dual basis $\{t^i\}\subseteq \fh^*$, then the co-bracket is simply
\be
\delta_r(t^i)=\sum f_{kj}^i t^j\otimes t^k=-\sum (-1)^{|x_j||x_k|}f^i_{jk}t^j\otimes t^k.
\ee
As in Definition \ref{def:grTakiff}, we give $\fh^*[-2]$ a shifted degree, and put the Lie algebra $\tgv$ over $\C\lbb\hbar\rbb$ where $\hbar$ has degree $-2$. Under this grading, the element $\delta_r$ is not degree preserving, but the element $\hbar\delta_r$ is. We are therefore interested in the quantization $(U(\tgv\lbb\hbar\rbb), \hbar\delta_r)$. In the following, we recall two quantizations of $(U(\tgv\lbb\hbar\rbb), \hbar\delta_r)$ via the work of Drinfeld. The equivalence of these two is a consequence of the canonical quantization scheme of \cite{etingof1996quantization}. 

\subsection{Drinfeld's quantization of Lie bi-algebras}

In this section, we introduce graded quantization of $\tgv^\hbar$ over the formal ring $\C\lbb\hbar\rbb$, according to \cite{drinfeld1986quantum, drinfeld1991quasi}. We introduce two quantizations, one as a quasi-triangular quasi-Hopf algebra, and one as quasi-triangular Hopf algebra. We then recall the results of Etingof-Kazhdan \cite{etingof1996quantization} on canonical quantization of Lie bi-algebras, from which it follows simply that these two quantizations give equivalent braided tensor categories over $\C\lbb\hbar\rbb$.  

\begin{Rem}

Note that  although \cite{drinfeld1986quantum, drinfeld1991quasi} as well as \cite{etingof1996quantization} only deal with ordinary finite-dimensional Lie bi-algebras and their quantizations, the constructions in these works extend without any effort to graded Lie algebras and Lie super-algebras, as long as one keeps the sign-rule for the graded commutator throughout. 

\end{Rem}

\subsubsection{Quantization as a quasi-Hopf algebra}\label{subsec:quant1}

We first recall the definition of quasi-triangular quasi-Hopf algebra over $\C\lbb\hbar\rbb$ by Drinfeld \cite{drinfeld1989quasi}.

\begin{Def}
A quasi-triangular quasi-Hopf algebra over $\C\lbb\hbar\rbb$ is a tuple $(A, \Delta^\hbar, \epsilon^\hbar, \Phi^\hbar, R^\hbar)$, where:
\begin{enumerate}
\item $A$ is a topological algebra over $\C\lbb\hbar\rbb$ such that $A/\hbar A$ is a standard universal enveloping algebra of some finite-dimensional Lie algebra $\fg$ and $A\cong V\lbb\hbar\rbb$ as a module of $\C\lbb\hbar \rbb$ where $V$ is a vector space. 

\item $\Delta^\hbar: A\to A\otimes_{\C\lbb\hbar \rbb} A$ is a continuous algebra homomorphism, $\epsilon^\hbar: A\to \C\lbb\hbar \rbb$ a continuous homomorphism such that $\epsilon^\hbar(1)=1, \Delta^\hbar(1)=1$. 

\item Invertible elements $R^\hbar\in   A\otimes_{\C\lbb\hbar \rbb} A$ and $\Phi^\hbar\in  A\otimes_{\C\lbb\hbar \rbb} A \otimes_{\C\lbb\hbar \rbb} A$. 

\item In the quotient $A/\hbar A$, the homomorphisms $\epsilon^\hbar, \Delta^\hbar$ becomes the standard symmetric Hopf algebra structure of $U(\fg)$, and $R=1 \mathrm{~mod~} \hbar$ and $\Phi=1 \mathrm{~mod~} \hbar$.

\end{enumerate}

These satisfy (omitting $\hbar$ from all equations to simplify notations):
\be
\begin{aligned}
(1\otimes \Delta) \Delta&=\Phi \cdot (\Delta\otimes 1)(\Delta)\cdot \Phi^{-1}\\
(1\otimes 1\otimes \Delta)(\Phi)\cdot (\Delta\otimes 1\otimes 1)(\Phi)&=(1\otimes \Phi)\cdot (1\otimes \Delta\otimes 1)(\Phi\cdot (\Phi\otimes 1)\\
(\epsilon\otimes 1)\circ \Delta &=1=(1\otimes \epsilon)\circ \Delta\\
(1\otimes \epsilon\otimes 1)(\Phi)& = 1\\
(\Delta)^{op}&=R\cdot \Delta\cdot R^{-1}\\
(\Delta\otimes 1)(R)&=\Phi^{312} R^{13}(\Phi^{132})^{-1}R^{23}\Phi,\\
(1\otimes \Delta)(R)&=(\Phi^{231})^{-1}R^{13}\Phi^{213}R^{12}\Phi^{-1}
\end{aligned}
\ee

\end{Def}

The above element $\Phi^\hbar$ is called a \textit{Drinfeld associator}. Given such a quasi-triangular quasi-Hopf algebra, one can consider the category of finite-flat modules of $A$, which we denote by $A\Mod$. By definition, an object in $A\Mod$ is a finitely-generated flat $\C\lbb\hbar\rbb$ module $M$, together with a continuous morphism:
\be
A\otimes_{\C\lbb\hbar\rbb} M\to M
\ee
making $M$ into a module of $A$. We have the following: 

\begin{Cor}
The category $A\Mod$ is a braided tensor category, where co-product is given by $\Delta^\hbar$, braiding is given by $R^\hbar$ and associativity given by $\Phi^\hbar$. 
\end{Cor}

Given the Lie algebra $\tgv^\hbar$, one can consider the universal enveloping algebra $U(\tgv^\hbar)=U(\tgv)\lbb\hbar \rbb$. This is a bi-algebra over $\C\lbb\hbar \rbb$ with standard co-product $\Delta_0$ and standard co-unit $\epsilon$. Define an element $R\in U(\tgv^\hbar)^{\otimes 2}$ by $R^\hbar =e^{\hbar \Omega}$, where $\Omega$ is from equation \eqref{eq:classicalC}. Note that $\Omega$ is the symmetrization of $r$, which is the classical $R$ matrix of equation \eqref{eq:classicalr}. 

\begin{Thm}[\cite{drinfeld1991quasi} Theorem A]
There exists an element $\Phi^\hbar\in  U(\tgv^\hbar)^{\otimes 3}$, unique up to twisting, such that $(U(\tgv^\hbar), \Delta_0, \epsilon, \Phi^\hbar, R^\hbar)$ is a quasi-triangular quasi-Hopf algebra. 

\end{Thm}

We comment that one can construct explicitly the element $\Phi^\hbar$ via solutions of \textit{Knizhnik-Zamolodchikov} (KZ) equations. More precisely, consider the following differential equation:
\be
G'(x)=\frac{1}{\pi i }\lp \frac{\hbar \Omega^{12}}{x}+\frac{\hbar \Omega^{23}}{x-1}\rp G(x)\in U(\tgv^\hbar)^{\otimes 3}. 
\ee
These have two unique solutions $G_1, G_2$ such that $G_1(x)\sim x^{\frac{\hbar}{\pi i} \Omega^{12}}$ as $x\to 0$, and $G_2(x)\sim (1-x)^{\frac{\hbar}{\pi i} \Omega^{23}}$ as $x\to 1$. The Drinfeld's associator is most easily defined by $\Phi^\hbar=G_2^{-1}G_1$, which is an element in $U(\tgv^\hbar)^{\otimes 3}$ such that $(U(\tgv^\hbar), \Delta_0, \epsilon, \Phi^\hbar, R^\hbar)$ is a quasi-triangular quasi-Hopf algebra.  This solution is denoted by $\Phi^\hbar_{KZ}$, since it arise from solutions of KZ equations. Consequently, the category $U(\tgv^\hbar)\Mod$ has the structure of a braided tensor category, where braiding is given by $R^\hbar=e^{\hbar \Omega}$, and associativity is given by $\Phi_{KZ}^\hbar$. This is also true for the category of graded modules, which we denote by $U(\tgv^\hbar)\Mod^{gr}$.

\begin{Rem}

It is in fact clear that the category $U(\tgv^\hbar)\Mod$ (as well as the graded version) has a ribbon twist, given by $\theta=e^{-\hbar C}$, where $C$ is the quadratic Casimir. 

\end{Rem}

\subsubsection{Quantization as Hopf algebra}\label{subsec:quant2}

Now we introduce another quantization, this time as a quasi-triangular Hopf algebra. This is simply a quasi-triangular quasi-Hopf algebra where $\Phi^\hbar=1$. The idea of this quantization is that $\tgv$ is the double $\tgv=D(\fh)$ as a Lie bi-algebra, and therefore a quantization of $U(\tgv)$ can be chosen to be the double of $U(\fh)$ as a Hopf algebra \cite{drinfeld1986quantum}. Since we are dealing with graded Lie algberas, we will use the double construction for graded Hopf algebras \cite{gould1993quantum}. 

The first problem is that $U(\fh)$ is not finite-dimensional, we therefore need to define it in the $\hbar$-adic sense. Let us consider the filtered commutative and co-commutative Hopf algebra $U(\fh^*[-2])\lbb\hbar\rbb$. As a Hopf algebra, this is simply isomorphic to the symmetric algebra $S(\fh^*[-2])\lbb\hbar\rbb$. We would like to define a dual pairing between $S(\fh^*[-2])\lbb\hbar\rbb$ and $U(\fh)$, via $(x_a, t^b)=\delta_a^b$. However, this is not graded, and to make this graded, we define $(x_a,\hbar t^b)=\delta_a^b$, and extend it by $\hbar$ linearity to all $S(\hbar\fh^*[-2])\lbb\hbar\rbb$. 

\begin{Lem}
The above extends to a pairing between $U(\fh)\lbb \hbar\rbb\otimes_{\C\lbb \hbar\rbb} S(\hbar \fh^*[-2])\lbb\hbar\rbb\to \C\lbb\hbar\rbb$. 
\end{Lem}

\begin{proof}
Using total symmetrizing map, we get an isomorphism $U(\fh)\lbb \hbar\rbb\to S(\fh)\lbb\hbar\rbb$, and therefore we obtain isomorphism:
\be
\Hom \lp S(\fh)\lbb\hbar\rbb, \C\lbb\hbar\rbb\rp=\overline{S}(\fh^*)\lbb\hbar\rbb\cong S(\hbar \fh^*[-2])\lbb\hbar\rbb.
\ee

\end{proof}

Therefore, we obtain a commutative Hopf algebra structure on $S(\hbar \fh^*[-2])\lbb\hbar\rbb$, defined by the following, for any $f, g\in S(\hbar \fh^*[-2]), a, b\in U(\fh)$: 
\be
f\cdot g(a)=f\otimes g(\Delta_0(a)),\qquad \Delta(f)(a, b)=f(a\cdot_{op} b)=(-1)^{|a||b|}f(b\cdot a).
\ee
In the above, $a\cdot_{op} b$ means the opposite multiplication. In other words, the Hopf algebra $S(\hbar \fh^*[-2])\lbb\hbar\rbb$ is defined to be the linear dual of $U(\fh)^{op}\lbb\hbar\rbb$ over $\C\lbb\hbar\rbb$. Using the $\C^\times$ grading, one can show that $\Delta(f)$ is a finite expression for any finite order of $\hbar$. The pairing between $U(\fh)\lbb\hbar\rbb$ and $S(\hbar \fh^*[-2])\lbb\hbar\rbb$ is a Hopf pairing. Therefore we can form the Drinfeld double over $\C\lbb\hbar\rbb$, which by definition is the matched product (for a detailed definition and formula, see \cite{agore2014classifying}):
\be
D(U(\fh)\lbb\hbar\rbb)=U(\fh)\lbb\hbar\rbb \match S(\hbar \fh^*[-2])\lbb\hbar\rbb,
\ee
 Note that since all the definitions of the matched pair involves at least one order of $\hbar$ for elements in $S(\hbar \fh^*[-2])\lbb\hbar\rbb$, one can simply multiply $\hbar^{-1}$ to all formulas. In particular, $(U(\fh)\lbb\hbar\rbb, S(\fh^*[-2])\lbb\hbar\rbb)$ is also a matched pair of Hopf algebras.

\begin{Def}
We define the graded algebra $\CA^\hbar$ to be the Hopf algebra:
\be
\CA^\hbar=U(\fh)\lbb\hbar\rbb \match S(\fh^*[-2])\lbb\hbar\rbb. 
\ee
\end{Def}

By Drinfeld's argument, this algebra is a quasi-triangular Hopf algebra, where the Hopf structure is given by the product of Hopf structures on $U(\fh)\lbb\hbar\rbb$ and $S(\fh^*[-2])\lbb\hbar\rbb$. The $R$ matrix is given by a formal sum over $S(\fh)\lbb\hbar\rbb$ with the dual basis, which is the unique such element representing the identity map on $U(\fh)$. More precisely, let $x_i$ be a set of basis for $\fh$ and $t^i$ dual basis in $\fh^*[-2]$, then $R$ has the form:
\be
R=\sum_{p}  p(x_i)p(\hbar t^i)=1+\hbar \sum x_i\otimes t^i+O(\hbar^2)\in U(\fh)\otimes S(\hbar\fh^*[-2])\lbb\hbar\rbb. 
\ee
In the above, $p$ runs over all polynomials over $\fh$. This $R$ satisfies:
\be
R(1\otimes a)=a,\qquad \forall a\in U(\fh)
\ee
where $R(1\otimes a)$ means we evaluate the second factor at $a$. 

In the following, we give very explicit formulas for this quasi-triangular Hopf algebra structure. 

\subsubsection{Explicit formula for the double}\label{subsubsec:trivcom}

We derive explicit commutation relations for the quantization as a double and its co-product, as well as a formula for the R-matrix. We will only derive the co-product up to the first order of $\hbar$, although it is not difficult to compute it to all orders of $\hbar$. As before, let us fix a set of graded basis $\{x_i\}$ for $\fh$ and corresponding dual basis $\{t^i\}\subseteq \fh^*[-2]$. We first work out the co-product of $S(\fh^*[-2])\lbb\hbar\rbb$. For this, consider elements in $U(\fh)\lbb\hbar\rbb$ of the form $e^{\hbar X}$ where $X\in \fh$. These form group-like elements in $U(\fh)\lbb\hbar\rbb$ such that:
\be
e^{\hbar X}e^{\hbar Y}=e^{H(\hbar X,\hbar Y)}
\ee
where $H(\hbar X,\hbar Y)=\hbar X+\hbar Y+\hbar^2/2[X,Y]+\cdots$ is the famous \textit{Baker-Cambell-Hausdorff} (BCH) formula. Clearly the co-product $\Delta$ is defined so that:
\be
\Delta (f)(e^{\hbar X}, e^{\hbar Y})= f(e^{\hbar X}\cdot_{op} e^{\hbar Y}),
\ee
We first show that such evaluation is enough to determine $f$. 

\begin{Lem}\label{Lem:evaluationatgroup}
Let $f\in S(\hbar\fh^*[-2])$, if $f(e^{\hbar X})=0$ for all $X\in \fh$, then $f=0$. 

\end{Lem}

\begin{proof}
Note that if we write $f=\sum_r p_r$ where $p_r$ is symmetric polynomials of degree $r$, then $f(e^{\hbar X})=\sum_r \hbar^rp_r(X)$, and therefore we only need to assume that $f=p_r$ for some symmetric polynomial of degree $r$. In this case, assume $X=\sum c^i x_i$, we have:
    \be
f(e^{\hbar X})=\hbar^r/r! (f, X^r)=\hbar^r/r! \lp f, \lp\sum c^i x_i\rp^r\rp. 
    \ee
    Note that we have:
    \be
\lp\sum c^i x_i\rp^r=\sum_{p}\sum_{k_1+\ldots +k_p=r} (c^{i_1})^{k_1}\cdots (c^{i_p})^{k_p}S(x_{i_1}, \cdots, x_{i_p}),
    \ee
    where $S(x_{i_1}, \cdots, x_{i_p})$ is a total symmetrization of $(x_{i_1})^{k_1}\cdots (x_{i_p})^{k_p}$. By the definition of the pairing, this means simply that $f(c_i)=0$ for all $c_i$, which means $f=0$. 

\end{proof}

We now view $t^i$ as elements in $S(\fh^*[-2])\lbb\hbar\rbb$ whose pairing with $U(\fh)\cong S(\fh)$ is so that $t^i$ only pair with $S^1(\fh)$. 

\begin{Cor}
The co-product on $\fh^*[-2]$ is given by the following, up to first-order in $\hbar$:
\be
\Delta (t^i)=t^i\otimes 1+1\otimes t^i-\frac{\hbar}{2}\sum (-1)^{|t^k||x_j|} f_{jk}^i t^j\otimes t^k+\CO(\hbar^2).
\ee 
 The co-unit is given by $\epsilon_\hbar (t^i)=0$ and anti-pode given by $S_\hbar(t^i)=-t^i$. In particular, $\CA^\hbar$ satisfies $\Delta-\Delta^{op}= [\Delta_0,\hbar r]+\hbar^2\CO(1)$, making $\CA^\hbar$ a quantization of $(U(\tgv), \hbar\delta_r)$.
\end{Cor}

\begin{proof}
The co-product follows from BCH formula, since:
\be
e^{\hbar Y}e^{\hbar X}=e^{\hbar X+\hbar Y+\frac{\hbar^2}{2}[Y,X]+\CO(\hbar^2)},
\ee
and 
\be
\Delta(t^i)(e^{\hbar X}, e^{\hbar Y})=t^i(\hbar X+\hbar Y+(-1)^{|X||Y|}\frac{\hbar^2}{2}[Y,X])+\CO(\hbar^2).
\ee
The result now follows since $[X,Y]=-(-1)^{|X||Y|}[Y,X]$. 

For the co-unit statement, by definition, it is linear dual to the unit, and therefore:
\be
\epsilon_\hbar (t^i)=t^i(1)=0,
\ee
since $t^i$ is linear dual to $\fh\subseteq U(\fh)$.

For the anti-pode statment, by definition, it is linear dual to the anti-pode in $U(\fh)$, which means:
\be
S_\hbar (t^i)(e^{\hbar X})=t^i(e^{-\hbar X})=-t^i(\hbar X)=-t^i(e^{\hbar X}).
\ee

\end{proof}

We now compute the commutator between $x_i$ and $t^j$.

\begin{Prop}
The commutation relation between $x_i$ and $t^j$ is given by:
\be
[t^j, x_i]=f_{ik}^j t^k,
\ee
and in particular, as an algebra, $\CA^\hbar\cong U(\tgv)\lbb\hbar\rbb$. 
 
\end{Prop}

\begin{proof}

By the definition of Drinfeld double, we have:
\be
(x_i\otimes 1)(1\otimes t^j)=x_i\otimes t^j,
\ee
where the opposite product is of the form (see \cite{gould1993quantum}):
\be
(1\otimes t^j)(x_i\otimes 1)=\sum (t^j)^{(1)}(Sx_i^{(1)}) (t^j)^{(3)}(x_i^{(3)}) x_i^{(2)}(t^j)^{(2)} \times (-1)^{(|t^j|+|x_i^{(1)}|)(|x_i|+|x_i^{(3)}|)}
\ee
Since $x_i$ is a primitive element, the above receives contribution from $3$ terms. The first one is:
\be
-(-1)^{(|t^j|+|x_i|)(|x_i|)} (t^j)^{(1)}(x_i)  (t^j)^{(2)} \epsilon (t^j)^{(3)}=-(-1)^{(|t^j|+|x_i|)(|x_i|)} (t^j)^{(1)}(x_i)  (t^j)^{(2)}.
\ee
The second one is:
\be
(-1)^{|t^j||x_i|}  x_i\otimes t^j.
\ee
The third one is:
\be
\epsilon (t^j)^{(1)}(t^j)^{(2)} (t^j)^{(3)}(x_i)=(t^j)^{(1)}(t^j)^{(2)}(x_i).
\ee
Let us sum up the first and third term, and pair this with an element $a\in U (\fh)$. The first term gives:
\be
\begin{aligned}
    -(-1)^{(|t^j|+|x_i|)(|x_i|)} &(t^j)^{(1)}(x_i)  (t^j)^{(2)}(a)\\ &=-(-1)^{|x_i|+|t^j||x_i|+|(t^j)^{(2)}||x_i|} (t^j)^{(1)}\otimes (t^j)^{(2)} (x_i\otimes a)\\ &=-(-1)^{|x_i|+|t^j||x_i|+|(t^j)^{(2)}||x_i|}(-1)^{|x_i||a|} t^j (ax_i).
\end{aligned}
\ee
Note that this is only nonzero when $|a|=|(t^j)^{(2)}|$, the above simplifies to
\be
-(-1)^{|x_i|+|t^j||x_i|} t^j (ax_i). 
\ee
On the other hand, the third term gives:
\be
(t^j)^{(1)}(a)(t^j)^{(2)}(x_i)= (-1)^{|a||(t^j)^{(2)}|}(t^j)^{(1)}\otimes (t^j)^{(2)} (a\otimes x_i)=(-1)^{|a||(t^j)^{(2)}|}(-1)^{|a||x_i|} t^j (x_ia). 
\ee
Again this is non-zero only when $|x_i|=|(t^j)^{(2)}|$ and so we are left with $t^j (ax_i)$. The sum of these two contributions is
\be
t^j (x_ia)-(-1)^{|x_i|+|t^j||x_i|} t^j (ax_i)
\ee
Now use the fact that the above is non-zero only when $|t^j|=|x|+|a|$, it simplifies to
\be
t^j (x_ia)-(-1)^{|x_i||a|}t^j (ax_i)=t^j([x_i, a])=(t^j\lhd x_i)(a). 
\ee
Here $\lhd$ means right conjugation action of $\fh$ on $S(\fh^*)$. When all dust settles, we find the following commutation relation:
\be
(1\otimes t^j)(x_i\otimes 1)-(-1)^{|t^j||x_i|} (x_i\otimes 1)(1\otimes t^j)=t^j\lhd x_i\otimes 1.
\ee
Since conjugation action of $\fh$ on $U(\fh)\cong S(\fh)$ preserves $S^n$, we see that $t^j\lhd x_i$ is equal to $\sum f_{ik}^j t^k$.

\end{proof}

We also have the following explicit formula for the $R$ matrix.

\begin{Prop}
The $R$ matrix of $\CA^\hbar$ is given by:
\be
R=\exp\lp\hbar \sum_i x_i\otimes t^i\rp=\exp \lp \hbar r\rp. 
\ee 

\end{Prop}

\begin{proof}

This statement comes from the fact that the $R$ matrix as an element in $U(\fh)\otimes S(\hbar\fh^*[-2])\lbb\hbar\rbb$ is the unique element representing the identity morphism $U(\fh)\to U(\fh)$. Namely we must have:
\be
R(1\otimes a)=a,\qquad \forall a\in U(\fh).
\ee 
We claim that $e^{\hbar r}$ satisfies the same relation. To prove this, by Lemma \ref{Lem:evaluationatgroup}, we only need to evaluate $e^{\hbar r}$ at $e^{\hbar X}$ for all $X\in \fh$. 

Note that $r$ satisfies $\hbar r(1\otimes X)=X$ by definition. Moreover, for any function $f, g\in S(\fh^*[-2])\lbb\hbar\rbb$, we have:
\be
f\cdot g(e^{\hbar X})=f(e^{\hbar X})g(e^{\hbar X})
\ee
since $\Delta(e^{\hbar X})=e^{\hbar X}\otimes e^{\hbar X}$. Therefore:
\be
e^{\hbar r}(1\otimes e^{\hbar X})=e^{\hbar r(1\otimes e^{\hbar X})}.
\ee
Now we note that $\hbar r(1\otimes e^{\hbar X})=\hbar r(1\otimes \hbar X)=\hbar X$, and therefore the RHS is $e^{\hbar X}$ and the proof is complete. 

\end{proof}

\begin{Rem}
All the quasi-triangular axioms are satisfied because of the properties of taking Drinfeld's double. As an illustration, let us explicitly show that $R\Delta R^{-1}=\Delta^{op}$. This is clear for generators in $\fh$, so let us take $f\in S(\fh^*[-2])\lbb\hbar\rbb$, and let $X, Y\in \fh$. We compute:
\be
R\Delta (f) R^{-1}(e^{\hbar X}, e^{\hbar Y})=R(e^{\hbar Y}) \Delta (f) R^{-1}(e^{\hbar Y}) (e^{\hbar X}, e^{\hbar Y}),
\ee
here we used the fact that $\Delta(e^{\hbar Y})=e^{\hbar Y}\otimes e^{\hbar Y}$, and we evaluate $R$ at $e^{\hbar Y}$ and act on the first factor of $\Delta(f)$ by conjugation. Since $R$ is simply the identity map $U(\fh)\to U(\fh)$, the above is equal to:
\be
(e^{\hbar Y}\otimes 1)\Delta (f)(e^{-\hbar Y}\otimes 1)(e^{\hbar X}, e^{\hbar Y})=\Delta(f)(e^{-\hbar Y}e^{\hbar X}e^{\hbar Y}, e^{\hbar Y}).
\ee
Now using the definition of $\Delta(f)$, we find:
\be
\Delta(f)(e^{-\hbar Y}e^{\hbar X}e^{\hbar Y}, e^{\hbar Y})=f(e^{\hbar Y}e^{-\hbar Y}e^{\hbar X}e^{\hbar Y})=f(e^{\hbar X}e^{\hbar Y})=\Delta^{op}(f)(e^{\hbar X},e^{\hbar Y}).
\ee

\end{Rem}

The fact that the quantization $\CA^\hbar\cong U(\tgv)\lbb\hbar\rbb$ as an algebra is a good sign, and in fact, it is a sign that this quantization is equivalent to the quantization as a quasi-Hopf algebra \cite[Section 3.2]{etingof1996quantization}. We will turn to this next. 

\subsection{Equivalence of the two quantizations}

In this section, we recall the work of Etingof-Kazhdan on canonical quantization of Lie bi-algebras \cite{etingof1996quantization}. In particular, we derive very simply from this work that the two quantizations given in Section \ref{subsec:quant1} and Section \ref{subsec:quant2} are equivalent. We record this in the following theorem.

\begin{Thm}\label{Thm:KLequiv}
There is an equivalence of braided tensor categories:
\be
U(\tgv^\hbar)\Mod\simeq \CA^\hbar\Mod
\ee
Moreover, this equivalence preserves grading and therefore induces an equivalence of the graded categories. The equivalence also intertwines the action of $\fg$ on objects and therefore induces an equivalence of the sub-category where the action of $\fg$ integrates to an algebraic action of $G$. 

\end{Thm}

\begin{Rem}
By transportation of structures, the algebra $\CA^\hbar$ has a ribbon element making the category $\CA^\hbar\Mod$ a ribbon category. 

\end{Rem}

\subsubsection{Canonical quantization of Lie bi-algebras}

Let $\fa$ be a Lie bi-algebra and $\fg=D(\fa)$ be its double with a classical $R$ matrix $r$. Let $\Omega=\frac{1}{2}(r+r^{21})$ be the Casimir. From such a double, Drinfeld constructed a canonical quantization of $U(\fg)$ over $\C\lbb\hbar\rbb$ as a quasi-Hopf algebra, whose definition was given in Section \ref{subsec:quant1}. This is a canonical quantization in the sense that its construction depends only on $\fa$ as a Lie bi-algebra. Essentially, the idea of \cite{etingof1996quantization} on the canonical quantization of Lie bi-algebras is to construct a fiber functor on Drinfeld's category, which then results in a quasi-triangular Hopf algebra, canonically quantizing $U(\fg)$ (in fact, isomorphic to $U(\fg)$), such that a quantization of $U(\fa)$ lies as a sub-algebra. Moreover, they show that $U(\fg)$ is a Drinfeld's double. 

More precisely, let $U(\fg)\lbb\hbar\rbb$ be the canonical quasi-triangular quasi-Hopf algebra, and in particular $U(\fg)\lbb\hbar\rbb\Mod$ is a braided tensor category. Let $M_+$ be the module $\mathrm{Ind}_{\fa}^{\fg}(\C)\lbb\hbar\rbb$ and $M_-$ be $\Ind_{\fa^*}^{\fg}(\C)\lbb\hbar\rbb$, which are two Verma modules associated to the two isotropic subalgebras. 

\begin{Thm}[\cite{etingof1996quantization} Section 2 \& 3]\label{Thm:canquant}
Let $F: U(\fg)\lbb\hbar\rbb\Mod\to \C\lbb\hbar\rbb\Mod$ be the functor:
\be
F(M)=\Hom_{U(\fg)\lbb\hbar\rbb}\lp M_+\otimes_{C\lbb\hbar\rbb} M_-, M\rp.
\ee
Then $F$ can be given the structure of an exact tensor functor. Let $H=\mathrm{End}(F)$ be the quasi-triangular Hopf algebra of endomorphisms of $F$. There is an isomorphism of algebras:
\be
U(\fg)\lbb\hbar\rbb\cong \End (F).
\ee

\end{Thm}

\begin{Rem}
The tensor product $M_+\otimes_{C\lbb\hbar\rbb} M_-$ is canonically isomorphic to $U(\fg)\lbb\hbar\rbb$ as a left module, and therefore the above fibre functor is given by taking Hom space with the regular representation. This is how one can identify the endomorphism algebra with $U(\fg)\lbb\hbar\rbb$. 

\end{Rem}

By transportation of relations, there exists a deformed co-product:
\be
\Delta_\hbar: U(\fg)\lbb\hbar\rbb\to U(\fg)\lbb\hbar\rbb\otimes_{C\lbb\hbar\rbb} U(\fg)\lbb\hbar\rbb,
\ee
which is given by a very explicit Drinfeld's twist. 

\begin{Prop}[\cite{etingof1996quantization} Lemma 3.3]\label{Prop:Jconj}
There exists an element $J\in U(\fg)^{\otimes 2}\lbb\hbar\rbb$ such that:
\be
\Delta_\hbar=J^{-1}\Delta_0 J,
\ee
where $\Delta_0$ is the symmetric co-product. More-over, the element $J$ is defined by a function of $\Omega^{ij}$ acting on $1_+\otimes 1_+\otimes 1_-\otimes 1_-$ as an element in $M_+^{\otimes 2}\otimes M_-^{\otimes 2}\cong U(\fg)^{\otimes 2}$. The anti-pode $S_\hbar$ is given by:
\be
S_\hbar(a)=Q^{-1}S(a)Q,
\ee
where $Q=\nabla (S\otimes 1)(J)$, $S$ is the original antipode of $U(\fg)$ and $\nabla$ is the multiplication map. The co-unit is un-deformed $\epsilon_\hbar=\epsilon$. 

\end{Prop}

Moreover, one can construct a pair of Hopf sub-algebras $U_\hbar (\fa), U_\hbar (\fa^*)$ of $U(\fg)\lbb\hbar\rbb$ using the modules $M_\pm$, which canonically quantizes $U (\fa)$ and $U (\fa^*)$. These appear as dual Hopf algebras inside $U(\fg)\lbb\hbar\rbb$, making it a Drinfeld double of either of the two Hopf subalgebras. 

\begin{Thm}[\cite{etingof1996quantization} Theorem 4.13]\label{Thm:candouble}
There are Hopf sub-algebras $U_\hbar (\fa), U_\hbar (\fa^*)$ of $U(\fg)\lbb\hbar\rbb$ quantizing the corresponding universal enveloping algebras. Moreover, the quasi-triangular Hopf algebra $U(\fg)\lbb\hbar\rbb$ is isomorphic to the double of $U_\hbar (\fa)$, such that the $R$ matrix coming from Theorem \ref{Thm:canquant} is identified with the $R$ matrix of the double. The Hopf algebras $U_\hbar(\fa), U_\hbar (\fa^*)$ and $U(\fg)\lbb\hbar\rbb$ quantizes the Lie bi-algebra structures on the Manin triple $(\fg, \fa, \fa^*)$. 

\end{Thm}

\subsubsection{Application to $\tgv$ and the equivalence}

We now apply this canonical quantization scheme to $\fa=\fh$ with trivial Lie co-algebra structure. Therefore, we obtain $U(\tgv)\lbb\hbar\rbb$ with a quasi-triangular Hopf algebra structure and is a double. By Theorem \ref{Thm:candouble}, to prove Theorem \ref{Thm:KLequiv}, we only need to show that $U_\hbar(\fh)=U(\fh)\lbb\hbar\rbb$ in our case. 

To do so, we must carefully examine the element $J$. Let $M_\pm$ be as above, we have natural identification:
\be
M_+\cong U(\fh^*)\lbb\hbar\rbb \qquad M_-\cong U(\fh)\lbb\hbar\rbb.
\ee
Denote by $1_\pm$ the unit in $M_\pm$, which is the generating element of the modules. As a left module of $U(\fg)$, the tensor product $M_+\otimes M_-$ can be naturally identified with $U(\fg)$ via the map:
\be
\phi: 1\to 1_+\otimes 1_-.
\ee
Denote by $\wt J$ the following map of $U(\tgv)$ modules:
\be
\wt J=\Phi_{1,2, 34}^{-1}(1\otimes \Phi_{2,3,4}) \tau e^{\hbar \Omega_{2,3}} (1\otimes \Phi_{2,3,4}^{-1})   \Phi_{1,2,34}
\ee
which is a map of $U(\tgv)$ modules $M_+^{\otimes 2}\otimes M_-^{\otimes 2}\to (M_+\otimes M_-)\otimes (M_+\otimes M_-)$. The element $J$ is given by:
\be
J=(\phi^{-1}\otimes \phi^{-1})\circ \wt J \lp 1_+\otimes 1_+\otimes 1_-\otimes 1_-\rp.
\ee
Note that $\wt J$ as an isomorphism of modules is built from an element in $U(\tgv)^{\otimes 4}\lbb\hbar\rbb$, and we will denote by $\wt J$ the same element. 

\begin{Lem}\label{Lem:xcommuteJ}
Write $\wt J$ as $\sum J^{(1)}\otimes J^{(2)}\otimes J^{(3)}\otimes J^{(4)}$, then for any $x\in \fh$, we have:
\be\label{eq:xJcommute}
0=[\Delta^4_0(x), \wt J 1_+\otimes 1_+\otimes 1_-\otimes 1_-],
\ee
Here $\Delta_0^4(x)=x^{(1)}\otimes x^{(2)}\otimes x^{(3)}\otimes x^{(x)}$, and the commutator is computed via identifications $M_+\cong U(\fh^*)\lbb\hbar\rbb$ and $M_-\cong U(\fh)\lbb\hbar\rbb$. 

\end{Lem} 

\begin{proof}

The element $\wt J$ is built from elements of the form $\hbar r^{i,j}$ for $i\ne j$ acting consecutively on $1_+\otimes 1_+\otimes 1_-\otimes 1_-$, and therefore, we only need to show that if we start with an element satisfying equation \eqref{eq:xJcommute}, then any application of $r^{i,j}$ on such an element still satisfies the same equation. Note that $r=\sum I_a\otimes I^a$, and note that because $\fh^*$ form an ideal in $\tgv$, the action of $I^a$ on $M_-$ is identically zero. We are left to consider the following two situations. 

First, assume we have an element $c^{(1)}1_-\otimes c^{(2)}1_+$ satisfying:
\be
[\Delta_0 (x), c^{(1)}1_-\otimes c^{(2)}1_+]=0,
\ee
then the element $r (c^{(1)}1_- \otimes c^{(2)}1_+)$ satisfies this equation as well. This follows from the fact that $[\Delta_0(x), r]=0$, and that $[\Delta_0(x), rc]=[\Delta_0(x), r]c+r[\Delta_0(x), c]$. 

Second, assume we have an element $c^{(1)}1_+\otimes c^{(2)}1_+$ satisfying:
\be
[\Delta_0 x, c^{(1)}1_+\otimes c^{(2)}1_+]=0,
\ee
we need to show that $r (c^{(1)}1_+ \otimes c^{(2)}1_+)$ satisfies this equation as well. Note that the difference from the first scenario is that the action of $I_a$ is not given by left multiplication but rather given by commutator. However, the same argument applies, since we have Jacobi identity $[\Delta_0(x), [r, c]]=[[\Delta_0(x), r], c]+[r, [\Delta_0(x), c]]$. 
 
\end{proof}

We derive from this the following:

\begin{Cor}
For any $x\in \fh$, let $J$ be the element supplied by Proposition \ref{Prop:Jconj}, then:
\be
J^{-1}\Delta_0(x)J=\Delta_0(x).
\ee
Consequently, $S_\hbar(x)=-x$. 

\end{Cor}

\begin{proof}

We need to show that:
\be
\Delta_0(x)J=J\Delta_0(x).
\ee
Let us consider the LHS. Since $\phi$ is a morphism of $\tgv$ modules, we have:
\be
\Delta_0(x)J=(\phi^{-1}\otimes \phi^{-1}) \lp x\wt{J} 1_+\otimes 1_+\otimes 1_-\otimes 1_-\rp,
\ee
We have by definition the following equality:
\be
\begin{aligned}
\lp x\wt{J} 1_+\otimes 1_+\otimes 1_-\otimes 1_-\rp &= [x, J^{(1)}1_+]\otimes J^{(2)}1_+\otimes J^{(3)}1_-\otimes J^{(4)}1_-\\ &+(-1)^{|J^{(1)}||x|}J^{(1)}1_+\otimes [x,J^{(2)}1_+]\otimes J^{(3)}1_-\otimes J^{(4)}1_-
\\ &+(-1)^{(|J^{(1)}|+|J^{(2)}|)|x|}J^{(1)}1_+\otimes J^{(2)}1_+\otimes xJ^{(3)}1_-\otimes J^{(4)}1_-\\ &+(-1)^{(|J^{(1)}|+|J^{(2)}|+|J^{(3)}|)|x|}J^{(1)}1_+\otimes J^{(2)}1_+\otimes J^{(3)}1_-\otimes xJ^{(4)}1_-.
\end{aligned}
\ee
Note that the difference between this and equation \eqref{eq:xJcommute} lies in the fact that the action of $x$ on $M_-$ is left multiplication rather than conjugation. However, using Lemma \ref{Lem:xcommuteJ}, the above is equal to:
\be
\begin{aligned}
    &(-1)^{|J^{(4)}||x|}J^{(1)}1_+\otimes J^{(2)}1_+\otimes (J^{(3)}1_-)x\otimes J^{(4)}1_-\\&+J^{(1)}1_+\otimes J^{(2)}1_+\otimes J^{(3)}1_-\otimes (J^{(4)}1_-)x\\ &= J (1\otimes 1\otimes \Delta_0(x)).
\end{aligned}
\ee
where we use right multiplication of $x$ on $M_-\cong U(\fh)$. The sign here is becasue $\sum_i |J^{(i)}|=0$ as $J$ is even. 

On the other hand, for any $a\in U(\fh^*)$ and $b\in U(\fh)$, we have:
\be
\phi(ab x)=\Delta_0(a)\Delta_0(b)\phi(x)=\Delta_0(a)\Delta_0(b)1_+\otimes x1_-=\Delta_0(a)1_+\otimes bx 1_-=a1_+\otimes bx1_-,
\ee
which is true because the action of $a$ on $M_-$ is identically zero. The above is in fact equal to:
\be
a1_+\otimes (b1_-)x=\phi(ab)x,
\ee
since $b\in U(\fh)$. Since multiplication $U(\fh^*)\otimes U(\fh)\to U(\tgv)$ is an isomorphism, we find that $\phi(ax)=\phi(a)x$ for any $a$. We now have:
\be
(\phi\otimes \phi)\lp J\Delta_0(x)\rp=\wt{J}(1_+\otimes 1_+\otimes 1_-\otimes 1_-)(\Delta_0(x)),
\ee
where the action of $\Delta_0(x)$ is right multiplication on $M_-\otimes M_-$. Therefore, the following is true:
\be
(\phi\otimes \phi)\lp J\Delta_0(x)\rp=(\phi\otimes \phi)\lp \Delta_0(x)J\rp,
\ee
which proves the statement for $\Delta_0(x)$. 

As for anti-pode, this is a simple consequence of the equalities:
\be
\nabla (S_\hbar \otimes 1)\Delta_\hbar(x)=\epsilon(x)=0,\qquad \Delta_\hbar(x)=\Delta_0(x)=x\otimes 1+1\otimes x,
\ee
which implies:
\be
S_\hbar(x)+x=0, \text{ i.e., } S_\hbar(x)=S(x)=-x. 
\ee
\end{proof}

Therefore, the ordinary embedding $U(\fh)\to U(\tgv)\lbb\hbar\rbb$ is an embedding of Hopf algebras. Consequently, by the definition of the embedding $U_\hbar(\fh)\to U(\tgv)\lbb\hbar\rbb$ according to \cite[Section 4]{etingof1996quantization}, we identify $U(\fh)\lbb\hbar\rbb$ with the image of $U_\hbar(\fh)$ as a Hopf algebra. Moreover, $ U(\tgv)\lbb\hbar\rbb$ as a Hopf algebra is isomorphic to the double of $U(\fh)\lbb\hbar\rbb$, namely $\CA^\hbar$. Clearly this equivalence intertwines the action of $\fg$ as $\fg\subseteq \fh$ and commutes with $J$. Since all the constructions in question respects the grading, this induces an equvialence of the category of graded modules. This finishes the proof of Theorem \ref{Thm:KLequiv}. 

\begin{Cor}

The quasi-triangular Hopf algebra $\CA^\hbar$ admits a ribbon element $\theta$.

\end{Cor}

\begin{proof}

Note that by \cite[Proposition 3.7]{etingof1996quantization}, the R-matrix $R$ of $\CA^\hbar$ is alternatively given by: 
\be
R=(J^{op})^{-1}e^{\hbar\Omega} J.
\ee
Therefore we have:
\be
R^{-1}(R^{op})^{-1}=J^{-1}e^{-2\hbar\Omega} J=J^{-1} \Delta_0(e^{-\hbar C})e^{\hbar C}\otimes e^{\hbar C}J=\Delta_\hbar(e^{-\hbar C})e^{-\hbar C}\otimes e^{-\hbar C}.
\ee
Here we used the fact that $C\otimes 1, 1\otimes C$ are central and therefore commute with $J$. We claim that $\theta=e^{-\hbar C}$ is the ribbon element. It is clearly central and satisfies $\Delta_\hbar(\theta)=R^{-1}(R^{op})^{-1} \theta\otimes \theta$, as well as $S(\theta)=\theta$ and $\epsilon(\theta)=1$. We only need to prove the following identity:
\be
\theta^2=uS(u),\qquad u=\nabla(S\otimes 1)(R).
\ee
Here $\nabla$ denote the multiplication map. Note that since $R=e^{\hbar r}$, we have:
\be
uS(u)=e^{ -\hbar\nabla(r)}e^{-\hbar S\nabla(r)}=e^{-\hbar\nabla(r)}e^{-\hbar \nabla(r^{op})}.
\ee
We now use the fact that $\nabla(r)+\nabla(r^{op})=2C$, as well as $[\nabla(r), \nabla(r^{op})]=0$ to conclude that the above is equal to $e^{-\hbar (\nabla(r)+\nabla(r^{op}))}=e^{-2\hbar C}=\theta^2$. This completes the proof.

\end{proof}

\section{Rozansky-Witten geometry of symplectic reduction}\label{sec: RWtangent}

\subsection{Symplectic reduction}

The action of $G$ on $V$ induces an action of $G$ on $T^*V$. The space $T^*V$ is symplectic with its canonical symplectic form. The action of $G$ on $T^*V$ preserves the symplectic form, and there is a moment map:
\be
\mu: T^*V\to \fg^*
\ee
such that:
\be
\mu(v,v^*)(X)=(Xv,v^*),~\text{ for all } X, v, v^*.
\ee
The map $\mu$ is a $G$-equivariant and generically flat, but in general not smooth. The zero fibre $\mu^{-1}(0)$ is a $G$-invariant subspace. Note that when $\mu$ is not flat, we take $\mu^{-1}(0)$ to be the \textbf{derived fiber}, namely the derived (affine) scheme represented by the following Cartesian diagram:
\be
\btik
\mu^{-1}(0)\rar\dar & 0\dar\\
T^*V\rar{\mu} & \fg^*
\etik
\ee 

\begin{Def}\label{Def:HamRed}
The symplectic reduction $\HR$ is defined as the affine scheme:
\be
\HR:=\mu^{-1}(0)/\!/G,
\ee
namely $\HR$ is an affine variety whose space of algebraic functions is the algebra $\C[\mu^{-1}(0)]^G$. We also denote the above by $T^*V\hypquotient G$, called the hyper-K\"{a}hler quotient. 
\end{Def}

\begin{Rem}\label{Rem:Cmu}
The affine scheme $\mu^{-1}(0)$ is represented by the following explicit commutative differential graded algebra. The underlying algebra is the algebra of functions on $T^*V\oplus \fg^*[-1]$. The differential is defined on generators of $\C[\fg^*[-1]]$ by $d\mu$. Explicitly, let $\{c_a\}$ be a set of linear coordinate functions on $\fg^*[-1]$, then the differential is given by:
\be
d c_a=(x_av, v^*). 
\ee
This DG algebra has a natural action of $G$ and $\C[\HR]$ is the $G$-invariant sub-algebra. 

\end{Rem}

The space $\HR$ has a Poisson structure inherited from $T^*V$, but it is usually not smooth. One usually defines a variantion of this space via the choice of a stability condition, and under certain conditions, the variation will be smooth. More precisely,  let $\xi\in \fg^*$ be a character of $G$, namely:
\be
\xi\in \Hom(G, \C^\times),
\ee
A point $p\in T^*V$ is called $\xi$-\textbf{semistable} if there exists $m\in \Z_{>0}$ and a function $f\in \C[T^*V]^{G, m\xi}$ such that $f(p)\ne 0$. Here $f\in \C[T^*V]^{G, m\xi}$ means that $f$ transforms under the action of $G$ as $m\xi$. Denote by $(T^*V)_{\xi}^{ss}$ the subset of all semi-stable points, which is an open subvariety of $T^*V$. Now for any subset $S\subseteq T^*V$, we define an equivalence relation on $S$ by declaring that $p\sim q$:
\be
\overline{G\cdot p}\bigcap \overline{G\cdot q}\bigcap S\ne \emptyset.
\ee

\begin{Def}
We define the GIT quotient variety $\HR^{\xi}$ associated to the stability condition $\xi$ by the quotient:
\be
\HR^{\xi}:= \left(\mu^{-1}(0)\cap (T^*V)_{\xi}^{ss}\right)/\sim.
\ee 
We also denote this by $T^*V\hypquotient_\xi G$
\end{Def}

\begin{Rem}\label{Rem:Proj}

We can alternatively define $\HR^{\xi}$ as the following projective variety, which is much more user-friendly:
\be\label{eq:projHiggs}
\HR^{\xi}:=\Proj (\bigoplus\limits_{n\geq 0} \mathbb{C}[\mu^{-1}(0)]^{G, m\xi}). 
\ee
Restriction of global functions give a map $\HR^\xi\to \HR$.
\end{Rem}

\begin{Rem}

Another way to think about the map $\HR^\xi\to \HR$ is that $\HR$ parametrizes closed $G$-orbits in $\mu^{-1}(0)$, and this map sends an orbit containing a semi-stable point to the closure. Again, there is a Poisson structure on  $\HR^{\xi}$ that is induced from $T^*V$, such that $\HR^\xi\to \HR$ is a map of Poisson varieties.

\end{Rem}

The variety $\HR^{\xi}$ is not smooth in general either, but in many cases, the variety $\HR$ has conical symplectic singularities and $\HR^\xi$ for a generic choice of $\xi$ gives rise to a symplectic resolution of $\HR$. In this case, one can completely characterize its deformation quantizations, and use it to study representation theory of non-commutative algebras. See for example, \cite{braden2016conical, braden2016conical2, losev2016deformations}. When $G=(\C^\times)^r$ is an abelian group, the following result also characterizes in which situation $\HR^{\xi}$ is a smooth variety. Define $(T^*V)_\xi^s$ the subset of $(T^*V)_\xi^{ss}$ where the stabilizer of $p$ is a finite subgroup. Then there is a cone $\Delta(G, T^*V)$ in $\fg^*$ defined by the property that for any $\xi\in \Delta(G, T^*V)$, $(T^*V)_\xi^{ss}=(T^*V)_\xi^s$, namely, stability and semi-stability agree. 

\begin{Thm}[\cite{hausel2002toric} Proposition 6.2; see also \cite{bellamy2012deformation} Corollary 4.13]\label{Thm:smoothstability}

If $\xi$ is in the interior of $\Delta(G, T^*V)$ and the action of $G$ on $T^*V$ is defined by a unimodular matrix over $\Z$, then $\HR^{\xi}$ is smooth. In this case, the map $\HR^{\xi}\to \HR$ is a resolution of singularities.  

\end{Thm} 

In this work, we are not concerned with the variety $\HR^\xi$, but rather with the stacky version of the symplectic reduction. Namely, we denote by $\sHR$ the stacky quotient $[\mu^{-1}(0)/G]$. By construction, the affine variety $\HR$ is defined so that $\C[\HR]$ is the algebra of global functions on $\sHR$. Moreover, various versions of the GIT quotient $\HR^\xi$ can be viewed as open sub-stacks of $\sHR$. This means in particular that we can study $\HR^\xi$ by studying $\sHR$. We will show in the next sections that when considering the category of coherent sheaves and the Poisson structure, the stack $\sHR$ is a much easier object than $\HR^\xi$. When the setting is clear, we will also omit the subscripts and simply denote the stack by $[\CM]$, and various variations by $\CM, \CM^\xi$. 

There is an action of $\C^\times$ on $\CM$ and $[\CM]$, under which $T^*V$ are of weight $-1$ and $\fg^*$ is of weight $-2$.  Under such a grading, the algebra $\C[T^*V]$ as well as $\C[\mu^{-1}(0)]$ are positively graded. This action clearly extends to $\CM^\xi$.

\subsection{Tangent Lie algebra of symplectic reduction}\label{subsec:symquotangent}

\subsubsection{Tangent Lie algebra of stacks}

Since the work of \cite{kapranov1999rozansky}, it is known that the shifted tangent complex $T_X[-1]$ of a smooth complex variety $X$ has the structure of a Lie algebra (or better, an $L_\infty$ algebra) in the symmetric monoidal category $\mathrm{Coh}(X)$ (the bounded derived category of coherent sheaves on $X$). The Lie algebra structure is given by the Atiyah class. Moreover, every object in $\Coh(X)$ comes equipped canonically with an action of $T_X[-1]$. 

When $X$ is symplectic, the Lie algebra $T_X[-1]$ also admits an invariant bilinear form, namely the symplectic form, which is naturally a map:
\be
T_X[-1]\otimes T_X[-1]\to \CO_X[-2].
\ee
Dualizing this map, one defines a Poisson bi-vector:
\be
\Omega_X: \CO_X[-2]\to T_X[-1]\otimes T_X[-1].
\ee
It was explained in \cite{kapustin2009three, kapustin2010three}, and later in \cite{roberts2010rozansky} that Rozansky-Witten theory with target $X$ is controlled by $T_X[-1]$ as a metric Lie algebra object in $\Coh (X)$. There are many layers of structures for which this is true, and we will focus on the structure of the braided tensor category this theory assigns. For any object $M, N\in \Coh (X)$, the action of $T_X[-1]$ together with $\Omega_X$ gives rise to an element in $\End^2(M\otimes N)$:
\be
\btik
M\otimes N[-2]\rar{\Omega_X} &  T_X[-1]\otimes T_X[-1]\otimes M\otimes N\rar{\tau_{23}} & T_X[-1]\otimes M\otimes T_X[-1]\otimes N\rar{\alpha_M\otimes \alpha_N} & M\otimes N
\etik
\ee
Here $\tau_{23}$ denotes flipping 2-3 entries, and $\alpha_M, \alpha_N$ are the Atiyah class of $M, N$, namely the module structure of $M, N$ as $T_X[-1]$ module. We will still denote this morphism by $\Omega_X$, or specify $\Omega_X(M, N)$ if necessary. The following was explained in \cite{roberts2010rozansky}, using the work of \cite{le1997parallel} on Kontsevich integrals and universal Vassiliev invariants. 

\begin{Thm}[ \cite{roberts2010rozansky} Theorem 9.1]\label{Thm:BTCRW}
Let $\Coh (X\lbb\hbar\rbb)$ be the category of coherent sheaves of $X\lbb\hbar\rbb$, where $\hbar$ is in cohomological degree $-2$. Then $\Coh (X\lbb\hbar\rbb)$ has naturally the structure of a braided tensor category, such that:

\begin{itemize}

\item The tensor product is given by the ordinary derived tensor product of sheaves.

\item The braiding is given by $\tau\circ e^{\hbar \Omega_X}$, which is an element in $\End^0(M\otimes N)$.

\item The associativity isomorphism $\Phi$ is given by Drinfeld's universal solutions of KZ equations. 

\end{itemize}

\end{Thm}

This braided tensor category is proposed to be the braided tensor category associated to Rozansky-Witten TQFT. The question was raised in \cite{roberts2010rozansky} of whether there is a twisting similar to \cite{etingof1996quantization} making the associativity isomorphism trivial. In the next section, we will answer this question positively for the stack $[\CM]$. We note that this also implies the similar statment for $\CM^\xi$ since the category $\Coh (\CM^\xi)$ is a DG quotient of $\Coh ([\CM])$ by a tensor ideal (namely the ideal generated by sheaves supported on the non-semistable loci). But before doing so, one must be able to extend the theory of tangent Lie algebras to stacks. In the remainder of this section, we recall the theory of tangent Lie algebras for stacks developed by \cite{hennion2018tangent} for derived Artin stacks, and in  \cite{gaitsgory2017study} for general pre-stacks locally almost of finite type. 

\begin{Rem}
We must warn the reader that the theory of tangent Lie algebras developed in \cite{hennion2018tangent} and  \cite{gaitsgory2017study} uses the theory of infinite categories, and as such goes beyond what is considered in this paper (abelian categories and the derived categories thereof). The powerful machinery of $\infty$-categories allows us to simply restrict the results there to the homotopy category, which are more classical derived categories. 

\end{Rem}

\begin{Rem}

We will follow more closely  \cite{gaitsgory2017study}, although it is equivalent to the theory of \cite{hennion2018tangent} when restricted to derived Artin stacks. In the process of introducing this theory, many definitions will be omitted simply because introducing them will side-track us from the main objectives of this paper. All definitions can be found in \cite{gaitsgory2017study}. We only need to know that the theories of tangent Lie algebras applies to $[\CM]$. 

\end{Rem}

For a stack $X$ that is locally almost finite type, the work of \cite{gaitsgory2017study} developed a theory of ind-coherent sheaves $\IndCoh (X)$. We will not go into detail about this category, since in our applications, the tangent Lie algebra will be an object in $\Perf(X)$.  Denote by $\mathrm{Lie}_X$ the category of DG Lie algebras in $\IndCoh (X)$, and by $\mathrm{PSt}_X^f$ the category of pointed formal stacks over $X$. By a pointed formal stack over $X$ we mean the following:
\be
\btik
Y \arrow[r, shift left, rightharpoonup, "\pi"] & \arrow[l, shift left, rightharpoonup, "s"]X
\etik
\ee
where $\pi$ is an inf-schematic nil-isomorphism, namely the restriction of $\pi$ on the reduced stack $Y^{\mathrm{red}}$ is an isomorphism, and $s$ is a section of $\pi$ such that $\pi\circ s=\mathrm{Id}$. Denote by $\mathrm{Gr}_X^f$ the category of formal groups over $X$, namely the category of group objects in $\mathrm{PSt}_X^f$. The first important result in \cite{gaitsgory2017study} is that there is a continuous functor $\Omega_X: \mathrm{PSt}_X^f\to \mathrm{Gr}_X^f$ given by:
\be
\Omega_X Y=X\times_Y X,
\ee
the derived intersection of $X$ in $Y$. This functor is an equivalence of categories, with inverse given by the Bar-complex of a group.  Moreover, there is a functor $\CL_X: \mathrm{Gr}_X^f\to \mathrm{Lie}_X$, which is essentially taking the Lie algebra of the formal group, whose adjoint $\mathrm{Exp}_X:  \mathrm{Lie}_X\to \mathrm{Gr}_X^f$ makes this adjoint pair equivalences of derived categories. We will un-pack this definition in more down-to-earth language. The tangent Lie algebra is defined to be the image of diagonal. 

\begin{Def}
Define the tangent Lie algebra $T_X[-1]$ of $X$ to be:
\be
T_X[-1]:= \CL_X \lp \Omega_X \lp \wh{X\times X}\rp\rp,
\ee
where $\wh{X\times X}$ is the formal completion of $X\times X$ along diagonal $\Delta: X\to X\times X$. 

\end{Def}

Let us now un-pack the definition a bit. If $\pi: G_X\to X$ is a formal group object over $X$, namely one is given a formal group law:
\be
m: G_X\times_X G_X\to G_X.
\ee
Let $\omega_{G_X}$ be the dualizing sheaf of $G_X$, then $U=p_*(\omega_{G_X})$ can be given the structure of a co-commutative bi-algebra in $\IndCoh(X)$. The co-algebra structure is given by the fact that $\omega_{G_X}$ is the dualizing sheaf, and the algebra structure is induced by $m$. Roughly speaking, one should think of $U$ as the universal enveloping algebra of the Lie algebra of $G_X$, such that $\CL_X(G_X)$ is precisely the subset of primitive elements of $U$. One can think of the filtered dual $U'$, which is a commutative bi-algebra, as the algebra of functions on $G_X$ (or better, a limit of rings representing $G_X$). Under such an identification, the formal group law $m: G_X\times_X G_X\to G_X$ is represented by the co-product on $U'$. If $G_X$ is classical and locally a fiber bundle over $X$, then the above is not only heuristic, but in fact rigorous. Namely, in this case, the Lie algebra $\fg_X:=\CL_X(G_X)$ is a classical Lie algebra over $X$ and is a vector bundle over $X$. Furthermore, one can identify $G_X$ as a formal scheme over $X$ with $\wh{\fg_X}$, the formal completion of $\fg_X$ at $0$ (the trivial fiber), whose group law is given by Baker-Campbell-Hausdorff formula. When $X$ itself is smooth and classical, $T_X[-1]$ is a vector bundle in degree $1$, and coincides with the tangent Lie algebra defined in \cite{kapranov1999rozansky}. In this case, the relation between tangent Lie algebra and loop space has been explored in \cite{ben2012loop}. 

This definition of tangent Lie algebra is functorial, namely that when $f: X\to Y$ is a morphism, then one obtain a natural map $f: T_X[-1]\to f^*T_Y[-1]$. At the level of formal stacks it is the map:
\be
1\times f: \wh{X\times X}\to \wh{X\times Y}
\ee
where $\wh{X\times Y}$ is the formal completion of $X\times Y$ at the image of $f$. If $f$ is an open embedding, then it is clear that $f$ induces a quasi-isomorphism of DG Lie algebras. Therefore, we obtain a tangent Lie algebra $T_{[\CM]}[-1]$ on the symplectic stack $[\CM]$, whose restriction over $\CM^\xi$ gives the tangent Lie algebra of this smooth variety (when it is smooth). We will show later that the tangent Lie algebra of the stack $[\CM]$ is in fact easier to compute than $\CM^\xi$. It seems unlikely at this stage since the definition of $T_{[\CM]}[-1]$ involves a sequence of non-trivial functors, all abstractly defined, whereas the definition for $T_{\CM^\xi}[-1]$ uses explicitly the Atiyah class. We will record the following lemma, which will be used to compute the tangent complex of $[\CM]$. Let $l_X$ be a Lie algebra over $X$, and for the purpose of application suppose it is a DG Lie algebra object in $\Perf(X)$. We denote by $\mathrm{CE}^*(l_X)$ the Chevalley-Eilenberg cochain complex of $l_X$, which is a commutative  DG algebra over $X$. The underlying commutative algebra is simply:
\be
\mathrm{CE}^*(l_X)=\mathrm{Sym}^*\lp l_X^*[-1]\rp
\ee
whose differential is given by Chevalley-Eilenberg differential. 

\begin{Lem}\label{Lem:affineTangent}

Let $\mathfrak{l}$ be a finite-dimensional DG Lie algebra (or more generally $L_\infty$ algebra) concentrated in positive degrees, and let $X=\mathrm{Spec}(\mathrm{CE}^*(\mathfrak{l}))$. Denote by $l_X$ the Lie algebra object given by $\mathrm{CE}^*(\mathfrak{l})\otimes \mathfrak{l}$, where the action of $\mathfrak{l}$ on itself is by conjugation. Then we have canonical identification:
\be
T_X[-1]\cong l_X.
\ee

\end{Lem}

\begin{proof}

Since $X$ is affine, we only need to show that $\wh{X\times X}$ as an object in $\mathrm{PSt}_X^f$ is represented by $\mathrm{CE}^*(l_X)$.  This is clear since we have a canonical isomorphism:
\be
\C[X\times X]=\mathrm{CE}^*(\mathfrak{l})\otimes \mathrm{CE}^*(\mathfrak{l})\cong \mathrm{CE}^*(l_X)
\ee
compatible with the projection along the first factor, given by:
\be
x_a\otimes 1\mapsto x_a, \qquad 1\otimes x_a-x_a\otimes 1\mapsto dx_a
\ee
where $\{x_a\}$ is a basis for $\mathfrak{l}$ and $dx_a$ represents sections of $l_X$. 

\end{proof}

\subsubsection{Poisson structure and braided tensor category}

We have now seen that for a stack $X$, one can define its tangent Lie algebra $T_X[-1]$, which is functorial with respect to open embeddings and which reduces to the one defined by \cite{kapranov1999rozansky} when $X$ is smooth. Applied to $X=[\CM]$, we obtain the tangent Lie algebra $T_{[\CM]}[-1]$, whose restriction gives rise to tangent Lie algebras over $\CM^\xi$. Since $[\CM]$ is a symplectic stack \cite{safronov2016quasi}, the tangent Lie algebra $T_{[\CM]}[-1]$ also carries a symplectic form, giving rise to a section:
\be
\omega: \CO\longrightarrow T_{[\CM]}[-1]\otimes T_{[\CM]}[-1].
\ee
The last ingredients to define a braided tensor category structure on $\Coh ([\CM])$ similar to Theorem \ref{Thm:BTCRW} is a canonical action of $T_{[\CM]}[-1]$ on objects in $\Coh ([\CM])$. From the definition of $T_{[\CM]}[-1]$, such an action simply follows from adjunctions. More precisely, let $p: \Omega_X(\wh{X\times X})\to X$ be the canonical projection, then by definition the functor $p_*p^!$ as a monoid can be identified with $U(T_X[-1])$, and the canonical adjunction:
\be
p_*p^!M\to M
\ee
gives rise to an action of $U(T_X[-1])$ on $M$, therefore an action of $T_X[-1]$ on $M$. Combining all these ingredients, we can extend Theorem  \ref{Thm:BTCRW} to $[\CM]$.

\begin{Thm}

Let $[\CM]=[\mu^{-1}(0)/G]$. Then the category $\Coh ([\CM]\lbb\hbar\rbb)$ has the structure of a braided tensor category, where the braiding is given by $\tau\circ e^{2\pi i \omega}$ and the associativity is given by the universal solution of KZ equation. 

\end{Thm}

In the next section, we will comput $T_{[\CM]}[-1]$ explicitly and show that it is related to $\tgv$. In particular, it will be represented by a perfect complex. This allows one to restrict the above to the full sub-category $\Perf ([\CM]\lbb\hbar\rbb)$ and various equivariant versions of the category. 

\section{Symplectic reduction and tangent Lie algebras}\label{sec:SRandTakiff}

\subsection{A DG model for the tangent Lie algebra}

\subsubsection{Tangent Lie algebra of fibers}\label{subsubsec:tangentfiber}

In this section, we compute the tangent Lie algebra of $[\CM]$ explicitly. This computation seems to be very familiar to experts, and can be deduced from \cite[Corollary 6.4.33]{nuiten2018lie}. We choose to present a direct proof here for the completeness of the presentation. Let $V$ and $W$ be two finite dimensional vector spaces with a linear $G$ action. Let $f: V\to W$ be a $G$-equivariant morphism, and let $X=f^{-1}(0)$. We will show in this section how to compute the tangent Lie algebra of $[X/G]$ explicitly as a $G$-equivariant $L_\infty$ algebra. 

Fix basis $\{v_i\}\subseteq V$ and $\{w_i\}\subseteq W$, let $\rho_V: V^*\to \C[G]\otimes V^*$ be the map given by matrix co-efficients, which is simply the pull-back of linear functions on $V$ along the map $G\times V\to V$. This map extends to an algebra homomorphism 
$\mathrm{Sym}(V^*)\to \C[G]\otimes \mathrm{Sym}(V^*)$. We denote by $T_V[-1]$ the tangent Lie algebra of $V$, which is simply the module $\C[V]\otimes V[-1]$ with trivial Lie bracket. The Lie algebra $\C[V]\otimes \fg$ acts on $T_V[-1]$ by Lie algebra homomorphism, and therefore we can form the smashed product  $\C[V]\otimes \fg\# V[-1]$. The action also gives rise to a map $\fg\to T_V$ by mapping $\fg$ to global vector fields on $V$. In the given basis, this map is given by:
\be
x_a\mapsto \rho(x_a)^i{}_j v_i v_j^*. 
\ee
We treat this as a differential on the smashed product $\C[V]\otimes \fg\# V[-1]$, which gives it the structure of a DG Lie algebra. It is obviously $G$-equivariant, and we denote this DG Lie algebra in $\Coh ([V/G])$ by $l_{V/G}$. 

\begin{Prop}\label{Prop:TVG}

The Lie algebra object $l_{V/G}$ can be identified with the tangent Lie algebra $T_{[V/G]}[-1]$. 

\end{Prop}

We first show the following lemma.

\begin{Lem}\label{Lem:copV}
Let $V[-1]\rtimes \wh{G}$ be the formal group defined by the semi-direct product of $\wh{G}$ with $V[-1]$. The formal group law induces the following co-product on $V^*[1]$:
\be
\Delta v_i^*[1]= v_i^*[1]\otimes 1+\rho_{ij}\otimes v_j^*[1],
\ee
where $\rho_{ij}$ is the matrix element of the action of $\wh{G}$ on $V$. 

\end{Lem}

\begin{proof}
This is a simple consequence of the following formal group law:
\be
(v_1, g_1)\cdot (v_2, g_2)=(v_1+\rho(g_1)v_2, g_1g_2). 
\ee

\end{proof}

\begin{proof}[Proof of Proposition \ref{Prop:TVG}]

The tangent Lie algebra $T_{[V/G]}[-1]$ is by definition the Lie algebra associated to the formal group $[Z/G]$, where $Z$ is defined by the following Cartesian diagram:
\be
\btik
Z\rar\dar & V\dar{\Delta}\\
\widehat{G}\times V\rar & \widehat{V\times V}
\etik
\ee
Here $\wh{G}$ is the formal completion of $G$ at identity and $\wh{X\times X}$ is the formal completion of $X\times X$ along diagonal. The formal scheme $Z$ is represented by the following very explicit DG algebra:
\be
\C[Z]=\C[V]\otimes \mathrm{Sym}(V^*[1])\otimes \C[\widehat{G}],
\ee
where the differential is given by $dv_i^*[1]=\rho_i-v_i^*$, where $\rho_i$ is the image of $v_i^*$ under the map $\rho: V^*\to V^*\otimes \C[G]$. The underlying algebra is easily identified with the algebra of functions on $V\times (V[-1] \rtimes \wh{G})$, and the differential is setting the equation $gv_i-v_i=0$. We immediately see from the definition of $l_{V/G}$ that one can identify the filtered dual $U(l_{V/G})'$ with functions on $Z$. We are left to compute the formal group law. 

Note that the formal group law on $Z$ is given by the following multiplication:
\be
Z\times_V Z\to Z: (g_1, v)\times (g_2, v)\to (g_1g_2, v)
\ee
which makes sense because $g_1v=v, g_2v=v$ implies that $g_1g_2v=v$. The induced co-product on $\C[\wh{G}]$ is clearly dual to the formal group law of $\wh{G}$, we only need to compute the co-product on $v_i^*[1]$. Note that the equation $g_1g_2v-v=0$ can be re-written into:
\be
0=g_1g_2v-v=g_1(g_2v-v)+g_1v-v,
\ee
and therefore the co-product on $v_i^*[1]$ representing the formal group law is given by:
\be
v_i^*[1]\mapsto v_i^*[1]\otimes 1+\sum_j \rho_{ij}\otimes v_j^*[1].
\ee
Indeed, we have the following identity:
\be
\begin{aligned}
d(v_i^*[1]\otimes 1+\sum_j \rho_{ij}\otimes v_j^*[1])(g_1, g_2, v)&=(v_i^*, g_1v-v)+\rho_{ij}(g_1)(v_j^*, g_2v-v)\\ &=(v_i^*, g_1v-v)+(v_i^*, g_1(g_2v-v))\\&=(v_i^*, g_1g_2v-v),
\end{aligned}
\ee
and the last is equal to the co-product of $dv_i^*[1]$. Using Lemma \ref{Lem:copV}, one sees that the above is precisely the formal group law of $V[-1]\rtimes \wh{G}$. Since this formal group can be represented by $U(l_{V/G})'$, we conclude that the following diagram of DG algebras commute:
\be
\btik
U(l_{V/G})'\otimes_{\C[V]} U(l_{V/G})' \dar{\cong}&\lar  U(l_{V/G})'\dar{\cong}\\
\C[Z\times_V Z] & \lar \C[Z]
\etik
\ee
There are two consequences of this. First of all, the formal group $Z$ has a formal group law that is associative without higher homotopy. Secondly, it is represented by $U(l_{V/G})'$ as a commutative Hopf algebra, from which it is then by definition that $\CL_X([Z/G])$ is precisely $l_{V/G}$. This completes the proof. 

\end{proof}

Consider now the fiber of $f: V\to W$ at $0$. We have the following fiber product:
\be
\btik
{[f^{-1}(0)/G] }\rar \dar& {[\mathrm{pt}/G] }\dar\\
{[V/G]}\rar & {[W/G]}
\etik
\ee
and therefore, it is clear that the tangent Lie algebra of $[f^{-1}(0)/G]$ fits into the following fiber product:
\be
\btik
T_{[f^{-1}(0)/G]}[-1]\rar\dar & \fg\dar\\
T_{[V/G]}[-1]\rar & T_{[W/G]}[-1]
\etik
\ee
Here we pull-back all the tangent Lie algebras over $f^{-1}(0)$. Since $T_{[V/G]}[-1], T_{[W/G]}[-1]$ and $\fg$ are all semi-direct products with $\fg$, we find that the resulting DG Lie algebra is the semi-direct product of $\fg$ with $T_{f^{-1}(0)}[-1]$, together with the natural differential $\fg\to T_{f^{-1}(0)}[-1]$. 

Let us now consider the following $L_\infty$ algebra $l_{f^{-1}(0)}$ on $f^{-1}(0)$. As a sheaf it is given by $\C[f^{-1}(0)]\otimes (V[-1]\oplus W[-2])$. Let $\nabla^n f$ be the $n$-th total derivative of $f$ with respect to the given basis of $V$ and $W$, and we think of this as an element in $\mathrm{Sym}^n (V^*[1])\otimes W[-2]$, which we view as an element in $\mathrm{Sym}^n (l^*)[2-n]\otimes l$. 

\begin{Lem}\label{Lem:Tf}
This gives $l_{f^{-1}(0)}$ the structure of an $L_\infty$ algebra, and we can identify $l_{f^{-1}(0)}$ with the tangent Lie algebra of $f^{-1}(0)$.

\end{Lem}

\begin{proof}

To show that this indeed is an $L_\infty$ structure, we need to show that the dual of $\sum_n\nabla^nf$, which we denote by $D$, defines a square zero differential on $\mathrm{Sym}^*\lp l^\vee[-1]\rp$. As a sheaf. this is identified with:
\be
\C[f^{-1}(0)]\otimes \mathrm{Sym}(V^*)\otimes \mathrm{Sym}(W^*[1]),
\ee
such that the map $D$ is given on generators by:
\be
D(w_i^*[1])=\sum \nabla^n f_i(v)  (v^*)^n=f_i(v+v^*)-f_i(v),\qquad D \mathrm{Sym}(V^*)=0.
\ee
It is clear now that it is a differential that square to zero. Moreover, we see that the Chevalley-Eilenberg algebra $\mathrm{CE}(l_{f^{-1}(0)})$ is precisely functions on the formal completion of $f^{-1}(0)\times f^{-1}(0)$ along the diagonal. Therefore by taking loop space we get the formal group corresponding to $l_{f^{-1}(0)}$. This shows that $l_{f^{-1}(0)}$ is the tangent Lie algebra of $f^{-1}(0)$. 
\end{proof}

Consequently, we see that the tangent Lie algebra of $[f^{-1}(0)/G]$ is represented by the following complex:
\be
\C[f^{-1}(0)]\otimes \fg\to \C[f^{-1}(0)]\otimes V[-1]\to \C[f^{-1}(0)]\otimes W[-2],
\ee
where the first differential is given by mapping $\fg$ to global vector fields, the second is given by $f_*$. The Lie algebra structure is given by the smashed product of $\fg$ with $\C[f^{-1}(0)]\otimes (V[-1]\oplus W[-2])$ whose $L_\infty$ algebra structure is given by $\nabla^n f$. 

We finish this section with the following comments. The tangent Lie algebra $T_{[f^{-1}(0)/G]}$ fits into the following exact sequence of DG Lie algebras:
\be
\btik
T_{f^{-1}(0)}\rar & T_{[f^{-1}(0)/G]}\rar & \CO\otimes \fg
\etik
\ee
as the underlying Lie algebra is a smashed product. This has a very geometric interpretation. Denote by $X:=f^{-1}(0)$ and $\CX:=[f^{-1}(0)/G]$, then the formal loop group of $\CX$ is represented by $[Z/G]$, where $Z$ is defined by the following Cartesian diagram, similar to Proposition \ref{Prop:TVG}:
\be
\btik
Z\rar\dar & X\dar{\Delta}\\
\widehat{G}\times X\rar & \widehat{X\times X}
\etik
\ee
The loop group of $X$, namely $\Omega_X\wh{X\times X}$ maps into $Z$ via the following extension of this diagram:
 \be
\btik
\Omega_X\wh{X\times X}\rar\dar & Z\rar\dar & X\dar{\Delta}\\
e\times X\rar &\widehat{G}\times X\rar & \widehat{X\times X}
\etik
\ee
The cone of this map is precisely the formal group $\widehat{G}\times X$. 

This geometry helps us understand the action of $T_\CX[-1]$ on an object in $\Coh(\CX)$ in the following way. Given a module of $T_\CX[-1]$, one can always first restrict to a module of $T_X[-1]$. Since in our situation this is a DG Lie algebra concentrated in positive degrees, the functor of taking Chevalley-Eilenberg cohomology is an equivalence. We conclude that the structure of a module of $T_X[-1]$ is equivalent to the structure of a module of $\mathrm{CE}^*_X (T_X[-1])$, via the functor:
\be
M\to \mathrm{CE}^*_X (T_X[-1])\otimes_X M. 
\ee
This functor is identified with the functor $\pi_*\pi^*$ where $\pi: \wh{X\times X}\to X$. The invariants $\pi_*\pi^*(M)$ is now a module of $\CO_X\otimes \fg$ due to the above short exact sequence, and this action of $\fg$ is simply induced from the $G$-equivariant structure on $\pi_*\pi^*(M)$. These two structures uniquely determines the action of $T_\CX[-1]$ on $M$. We summarize this into the following lemma, which will be useful in the proof of Proposition \ref{Prop:CFaction}. 

\begin{Lem}\label{Lem:SESmodule}

Let $X=f^{-1}(0)$ and $\CX=[f^{-1}(0)/G]$. For any $M\in \Coh(\CX)$, the canonical action:
\be
T_\CX[-1]\otimes M\to M
\ee
is uniquely determined by the following two data:
\begin{enumerate}

\item The action morphism $T_X[-1]\otimes M\to M$, or equivalently the object $\pi^*M$ as a sheaf over $\wh{X\times X}$. 

\item The action of $\wh{G}$ on $\pi_*\pi^*(M)$. 

\end{enumerate}

\end{Lem}

\subsubsection{Application to symplectic reduction}

We now apply this to $\mu: T^*V\to \fg^*$. We start with the following simple observation. Recall that $\tgv^+$ is the subalgebra of $\tgv$ generated by elements in positive degrees. 

\begin{Lem}\label{Lem:Cmu=CE}
There is a quasi-isomorphism of DG algebras:
\be
\C[\mu^{-1}(0)]\cong \mathrm{CE}^*(\tgv^+).
\ee

\end{Lem}

\begin{proof}

The underlying commutative algebra of $\mathrm{CE}^*(\tgv^+)$ is the algebra $\mathrm{Sym}((\tgv^+)^*[-1])$, which is identified with functions on $T^*V\oplus \fg^*[-1]$. The Chevalley-Eilenberg differential is non-trivial only on $\C[\fg^*[-1]]$ (since the only non-trivial commutator here is $\{\psi_+^i,\psi_-^j\}=\sum \rho^{ij}(x_a)t^a$), we see that the Chevalley-Eilenberg differential is given on generators of $\fg^*[-1]$ by:
\be
dc_a=(x_av, v^*).
\ee
This coincides with $\C[\mu^{-1}(0)]$, as in Remark \ref{Rem:Cmu}.
\end{proof}

Let us treat now $\tgv$ as a module of $\tgv^+$ via conjugation, and consider the following object:
\be
l_\CM:= \mathrm{CE}^*(\tgv^+)\otimes \tgv,
\ee
as a DG module of $ \mathrm{CE}^*(\tgv^+)=\C[\mu^{-1}(0)]$. Clearly the Lie bracket of $\tgv$ is compatible with the differential and therefore this is a sheaf of Lie algebra over $\mu^{-1}(0)$. Moreover, it has a canonical $G$-equivariant structure, making it into a Lie algebra object in $\Coh (\CM)$. Explicitly, the differential is given by:
\be
dx_a=\sum_i\rho^{i}{}_j(x_a)(v_i^* \psi^{j}_+-v_j\psi^{i}_-),\qquad  d\psi^{i}_-=\sum_a \rho^j{}_{i}(x_a)t^av_j^*,\qquad d\psi^{i}_+=\sum_{a}\rho^i{}_{j}(x_a)t^a v_j.
\ee
Here $\{v_i^*\}$ is a set of coordinates on $V$ and $\{v_i\}$ is a set of coordinates on $V^*$. We claim:

\begin{Prop}\label{Prop:TM}
We can identify the tangent Lie algebra of $[\CM]$ with $l_\CM$. Under this identification, the Poisson form on $[\CM]$ is given by the quadratic form $\kappa$ on $\tgv$. 
\end{Prop}

\begin{proof}
Since $\tgv$ is a semi-direct product, we just need to show that the Lie algebra structure on $\tgv^+$ coincides with the one given in Lemma \ref{Lem:Tf}, which follows clearly from  Lemma \ref{Lem:affineTangent}. As for the symplectic form, it is simply a consequence of \cite[Example 2.2.1]{safronov2016quasi}. 

\end{proof}

\subsection{An equivalence of braided tensor categories}

As before, let $U(\tgv^\hbar)\Mod^{gr}_G$ denote the abelian category of $G$-equivariant graded modules of $U(\tgv^\hbar)$ which are finitely generated and flat over $\C\lbb\hbar\rbb$, and by $D^b U(\tgv^\hbar)\Mod^{gr}_G$ its bounded derived category. On the other hand, we define the category $\Perf_{\C^\times}([\CM]\lbb\hbar\rbb)$ as follows. We have seen in Lemma \ref{Lem:Cmu=CE} that $\mu^{-1}(0)$ as an affine variety is represented by $\mathrm{CE}^*(\tgv^+)$. We denote by $\Perf_{\C^\times}([\CM]\lbb\hbar\rbb)$ the full subcategory of the derived category of $G\times \mathbb{C}^\times$-equivariant modules of $\mathrm{CE}^*(\tgv^+)\lbb\hbar\rbb$ generated by objects of the form $\mathrm{CE}^*(\tgv^+)\lbb\hbar\rbb\otimes V$ where $V$ is some finite-dimensional $G$ module. 

The main goal of this section is the proof of the following statement.

\begin{Thm}\label{Thm:tangentBTC}
There is an equivalence of triangulated braided tensor categories:
\be
D^b U(\tgv^\hbar)\Mod^{gr}_G\simeq \Perf_{\C^\times}([\CM]\lbb\hbar\rbb). 
\ee

\end{Thm}

The proof of this statement will be divided into the following steps. 

\begin{enumerate}

\item We first construct an explicit functor $\CF$ between two sides, starting with a larger category which makes the equivalence manifest on the chain level. The proof follows almost word-to-word to the Koszul Duality proof in \cite{beilinson1996koszul}. 

\item We then upgrade the functor $\CF$ to an equivalence of symmetric monoidal categories, in a way that the action of the Lie algebra $\tgv$ on the LHS is identified with the action of tangent Lie algebra on the RHS.

\item Since by Proposition \ref{Prop:TM}, the quadratic Casimir on the LHS is identified with the Poisson bi-vector on the RHS, we can use Drinfeld's universal solution to KZ equations to show that the functor $\CF$ is an equivalence of braided tensor categories. 

\end{enumerate}

\subsubsection{Step 1: constructing the functor}

Let $\tgv^+$ be the Lie sub-algebra generated by elements of positive degrees. As we have seen, there is a quasi-isomorphism:
\be
\mathrm{CE}^*(\tgv^+)\cong \C[\mu^{-1}(0)],
\ee
from which we derived that $\tgv^+$ is the tangent Lie algebra of $\mu^{-1}(0)$. Consider the following graded object:
\be
\CB:= U(\tgv^+)\otimes \mathrm{CE}_*(\tgv^+) , \qquad d=\sum_i x_i \otimes \psi_+^i+y_i\otimes \psi_-^i+c_a\otimes t^a+d_{CE},
\ee
Then $\CB$ induces a pair of adjoint functors:
\be
\CF: U(\tgv^\hbar)\Mod^{gr}_G \rightleftharpoons \QCoh_{\C^\times}([\CM]\lbb\hbar\rbb) :\CG
\ee
where $\CF(V)=\Hom_{\tgv^+} (\CB, V)$, and $\CG(M)=\CB\otimes_{\mathrm{CE}^*(\tgv^+)} M$. We need, however, to correctly account for the grading, and this is can be done similarly to \cite{beilinson1996koszul}. Since $U(\tgv^+)$ is a positively graded ring, we define $C(\tgv^+\lbb\hbar\rbb)_G$ to be the homotopy category of complexes of graded $G$-equivariant modules of $\tgv^+\lbb\hbar\rbb$ that are flat over $\C\lbb\hbar\rbb$, where $\hbar$ is in equviariant degree $-2$ and homological degree $0$. An object in this category is given by a complex:
\be
\btik
\cdots M^i\rar{d_M} & M^{i+1}\rar{d_M}& M^{i+2}\cdots
\etik
\ee
such that each $M^i=\oplus_j M^i_j$ and $d_M$ is grading-preserving, and that each $M^i$ is a flat module over $\C\lbb\hbar\rbb$ (namely isomorphic to $V\lbb\hbar\rbb$ for some graded vector space $V$). We define $C(\tgv^+\lbb\hbar\rbb)_G^+$ be the full subcategory where the $M$ satisfies:
\be
M^i_j=0 \text{ if } i\ll 0 \text{ or } i+j\gg 0
\ee
On the other hand, denote by $C(\mathrm{CE}^*(\tgv^+)\lbb\hbar\rbb)_G^-$ the full subcategory of the homotopy category of the DG category of graded $G$-equivariant modules of $\mathrm{CE}^*(\tgv^+)\lbb\hbar\rbb$, flat over $\C\lbb\hbar\rbb$ such that:
\be
M^i_j=0 \text{ if } i\gg 0 \text{ or } i+j\ll 0
\ee
We give $\hbar$ equivariant degree $2$, but now also the cohomology degree $-2$. Denote by $D(\tgv^+\lbb\hbar\rbb)_G^+$ the derived category of $C(\tgv^+\lbb\hbar\rbb)_G^+$ (namely the localization over quasi-isomorphisms), and similarly $D(\mathrm{CE}^*(\tgv^+)\lbb\hbar\rbb)_G^-$ the derived category of $C(\mathrm{CE}^*(\tgv^+)\lbb\hbar\rbb)_G^-$. 

\begin{Thm}
There exists a pair of adjoint functors: 
\be
\CF: D(\tgv^+\lbb\hbar\rbb)_G^+ \rightleftharpoons  D(\mathrm{CE}^*(\tgv^+)\lbb\hbar\rbb)_G^- :\CG
\ee
that induce equivalences of triangulated categories. 

\end{Thm}

\begin{proof}

\noindent\textbf{Step 1.}

Define functor $\CF: C(\tgv^+\lbb\hbar\rbb)_G^+\to C(\mathrm{CE}^*(\tgv^+)\lbb\hbar\rbb)_G^-$ by:
\be
\CF(M)=\mathrm{CE}^*(\tgv^+)\otimes M, \qquad d_{\CF(M)}=d_M+ \sum_i x_i \otimes \psi_+^i+y_i\otimes \psi_-^i+c_a\otimes t^a+d_{CE}.
\ee
We define grading so that:
\be
\CF(M)^p_q=\bigoplus_{\substack{ p=i+j+k\\ q=l-j}} \mathrm{CE}^*(\tgv^+)^k_l\otimes M^i_j
\ee
One can then show, using the fact that $\mathrm{CE}^*(\tgv^+)$ is positively graded and that $\mathrm{CE}^*(\tgv^+)^k_l=0$ for $|k|\ll 0$ to show that $\CF(M)$ indeed lies in $C(\mathrm{CE}^*(\tgv^+)\lbb\hbar\rbb)_G^-$. The same argument of \cite[Theorem 2.12.1]{beilinson1996koszul} implies that this descends to the derived category. Define now $\CG:  C(\mathrm{CE}^*(\tgv^+)\lbb\hbar\rbb)_G^-\to C(\tgv^+\lbb\hbar\rbb)_G^+$ by:
\be
\CG(N)=\oplus_j \Hom_\C( U(\tgv^+)_j, N),\qquad d_{\CG(N)}=d_N+\sum_i x_i \otimes \psi_+^i+y_i\otimes \psi_-^i+c_a\otimes t^a. 
\ee
We define grading by:
\be
\CG (N)^p_q=\bigoplus_{\substack{p=i+j\\ q=l-j}} \Hom_\C(U(\tgv^+)_{-l}, N^i_j). 
\ee
This object clearly lies in $C(\tgv^+\lbb\hbar\rbb)_G^+$, and moreover $\CG$ descends to the derived category. 

\noindent\textbf{Step 2.} 

The functors $\CF$ and $\CG$ are adjoint to each other, thanks to the fact that these are defined by tensor-Hom adjunction of the object $\CB$. More precisely, we have canonical isomorphism of DG vector spaces:
\be
\Hom_{ \mathrm{CE}^*(\tgv^+)\lbb\hbar\rbb}(\CF(M), N)=\Hom_{\C\lbb\hbar\rbb} (M,N)=\Hom_{\tgv^+\lbb\hbar\rbb}(M, \Hom_{\C\lbb\hbar\rbb}(U(\tgv^+)\lbb\hbar\rbb, N)),
\ee
we note that $\Hom_{\C\lbb\hbar\rbb}(U(\tgv^+)\lbb\hbar\rbb, N)\cong \Hom_{\C}(U(\tgv^+), N)$ canonically. Moreover, if $N$ is such that $U(\tgv^+)_lN=0$ for large enough $l$, then:
\be
\Hom_{\tgv^+\lbb\hbar\rbb}(M, \Hom_{\C}(U(\tgv^+), N))=\Hom_{\tgv^+\lbb\hbar\rbb}(M, \bigoplus_l \Hom_{\C}(U(\tgv^+)_l, N)).
\ee
As in the proof of \cite[Lemma 2.12.2]{beilinson1996koszul}, this establishes the adjointness of $\CF$ and $\CG$. Since the above isomorphisms are all isomorphisms of $G$-modules, we see that this adjoint pair are for $G$-equivariant categories. 

\noindent\textbf{Step 3.} 

The canonical natural transform $\CF\circ \CG\to 1$ and $1\to \CG\circ \CF$ are both equivalences. This will follow word to word from the proof of \cite[Lemma 2.12.3, Lemma 2.12.4]{beilinson1996koszul}. Note that the flatness over $\C\lbb\hbar\rbb$ ensures that tensor products over $\C\lbb\hbar\rbb$ of $M$ or $N$ with any object can be computed by ordinary tensor product, and therefore does not present difficulties in repeating the proof of \textit{loc.cit}.

\end{proof}

We now wish to restrict to a full subcategory. Denote by $D^b(\tgv^+\lbb\hbar\rbb)_G$ the full subcategory of $D(\tgv^+\lbb\hbar\rbb)_G^+$ consisting of objects $M^*$ whose cohomology satisfies $H^i(M^*)=0$ for $|i|\gg 0$ and that $H^i(M^*)$ are finitely generated over $\C\lbb\hbar\rbb$. On the other hand, by the definition of the category $\Perf_{\C^\times}([\CM])$, it clearly sits as a full-subcategory of $D(\mathrm{CE}^*(\tgv^+)\lbb\hbar\rbb)_G^-$. We have the following corollary:

\begin{Prop}\label{Prop:equivbounded}
The functor $\CF$ and $\CG$ induces an equivalence of triangulated categories:
\be
\CF: D^b(\tgv^+\lbb\hbar\rbb)_G  \rightleftharpoons \Perf_{\C^\times}([\CM]) :\CG.
\ee

\end{Prop}

We first prove a lemma. We defined the category $ D^b(\tgv^+\lbb\hbar\rbb)_G$ as the full subcategory whose cohomology is finitely generated over $\C\lbb\hbar\rbb$, but we could have started with the homotopy category of finite chain complexes of finite-flat modules and invert quasi-isomorphisms. This category is precisely $D^b U(\tgv^\hbar)\Mod^{gr}_G$. There is of course a natural functor $D^b U(\tgv^\hbar)\Mod^{gr}_G\to D^b(\tgv^+\lbb\hbar\rbb)_G$ that is essentially surjective. We claim that this is an equivalence. 

\begin{Lem}
The natural functor $D^b U(\tgv^\hbar)\Mod^{gr}_G\to D^b(\tgv^+\lbb\hbar\rbb)_G$ is fully-faithful. 

\end{Lem}

\begin{proof}
Let $U(\tgv^+)\lbb\hbar\rbb$ be the universal enveloping algebra, which is a projective object in $C(\tgv^+\lbb\hbar\rbb)_G$. It is a standard fact that an object in $C(\tgv^+\lbb\hbar\rbb)_G$ that is quasi-isomorphic to a finite-flat module has a resolution by $U(\tgv^+)\lbb\hbar\rbb$. Indeed, let $V$ be a finite flat object in $C(\tgv^+\lbb\hbar\rbb)_G$, which for simplicity we assume to be in homological degree $0$. Then the morphism $U(\tgv^+)\otimes \mathrm{CE}_*(\tgv^+)\otimes V\to V$ is a $G$-equivariant projective resolution of $V$. Such a resolution computes homomorphisms in $D(\tgv^+)_G$. We simply need to show that these computes homomorphisms in $D^b U(\tgv^\hbar)\Mod^{gr}_G$ as well. 

Since $P^*:=U(\tgv^+)\otimes \mathrm{CE}_*(\tgv^+)\otimes V$ is a projective resolution of $V$ (and therefore final in the category of resolutions of $V$), for any other resolution $M^*\to V$, there must be a map $P^*\to M^*$. Let $I$ be the ideal of $U(\tgv^+)$ generated by $\{t^a\}$. If we assume that $M^*$ is also a finite complex of finite flat modules, then there must exist $N\gg 0$ such that the map $P^*\to M^*$ factors through $P^*/I^N$. Let $\tau_{\leq M}P^*$ be the cut-off of $P^*$:
\be
\tau_{\leq M}P^*:= \mathrm{Ker}(P^{-M})\to P^{-M}\to P^{-M+1}\to \cdots \to P^0
\ee
then the map $P^*\to M^*$ also factors through $\tau_{\leq M}P^*$ for some large $M$. Therefore, the projective system:
\be
\varprojlim\limits_{M,N}\tau_{\leq M} P^*/I^N
\ee
is final in the category of finite chain complexes of finite-flat modules, namely in finite chain complexes of the abelian category $U(\tgv^\hbar)\Mod^{gr}_G$. Now the lemma follows from the following two equalities, where $W^*$ is a finite chain complex of finite-flat modules. 
\be
\begin{aligned}
&\Hom^0_{C(\tgv^+\lbb\hbar\rbb)_G}(P^*, W^*)=\varinjlim\limits_{M', N'}\Hom^0( \tau_{\leq M'} P^*/I^{N'}, W^*), \\
&\Hom^0_{D^b U(\tgv^\hbar)\Mod^{gr}_G}(V^*, W^*)=\varinjlim\limits_{M', N'}\Hom^0( \tau_{\leq M'} P^*/I^{N'}, W^*)
\end{aligned}
\ee

\end{proof}

\begin{Cor}\label{Cor:closedcones}
The category $D^b(\tgv^+\lbb\hbar\rbb)_G$ is closed under taking cones. 

\end{Cor}

\begin{proof}

Let $f: M^*\to N^*[-n]$ be a morphism in $D^b(\tgv^+\lbb\hbar\rbb)_G$, then there exists a chain complex $\wt{M}^*$ quasi-isomorphic to $M^*$ such that $f: \wt{M}^*\to N^*[-n]$ is a morphism of chain complexes. The cone of $f$ is therefore represented by an object in $D^b(\tgv^+\lbb\hbar\rbb)_G$, namely $\wt{M}^*\oplus N^*[-n-1]$. 

\end{proof}

\begin{proof}[Proof of Proposition \ref{Prop:equivbounded}]

We only need to show that the functors map objects into the desired full subcategory. It is clear for $\CG$ by definition, as $\CG$ maps $\mathrm{CE}^*(\tgv^+)\lbb\hbar\rbb\otimes V$ to $V\lbb\hbar\rbb$ with trivial action of $\tgv^+$, and therefore belongs to $D^b(\tgv^+\lbb\hbar\rbb)_G$. Since by Corollary \ref{Cor:closedcones}, $D^b(\tgv^+\lbb\hbar\rbb)_G$ is closed under taking cones, $\CG$ maps any object in $ \Perf_{\C^\times}([\CM])$ to $D^b(\tgv^+\lbb\hbar\rbb)_G $. 

To show the statement for $\CF$, we may assume that $M$ is in a single homological degree and is finite over $\C\lbb\hbar\rbb$. Since the equivariant grading on $M$ is bounded from above, let $M_0$ be the top degree part of $M$, then by grading constraints, it is a $G$-equivariant submodule of $M$ on which $\tgv^+$ acts trivially. Therefore the module $M_0\lbb\hbar\rbb$ maps to a sub-object of $M$. Let $N:=M/M_0\lbb\hbar\rbb$, then it is finite over $\C\lbb\hbar\rbb$ generated by less generators than $M$. Inductively, we find that $M$ comes equipped with a finite filtration $M_i$ such that $M_i/M_{i+1}$ is finite over $\C\lbb\hbar\rbb$ and has trivial action of $\tgv^+$. We only need to show that such a module is mapped by $\CF$ to $ \Perf_{\C^\times}([\CM])$. 

We first assume $M_i/M_{i+1}$ is flat. Then in this case it is equal to $V\lbb\hbar\rbb$ for some finite-dimensional $G$-module $V$, and clearly the image of this under $\CF$ is in $ \Perf_{\C^\times}([\CM])$. Now suppose it is not flat, then WLOG, we may assume that it is concentrated in a single equivariant degree (by taking the associated graded of a finer filtration). Then it is clear that it is quasi-isomorphic to the complex $M_i/M_{i+1}\lbb\hbar\rbb\stackrel{\hbar}{\to}M_i/M_{i+1}\lbb\hbar\rbb$, and now the statement follows from the corresponding statement for flat modules. 

\end{proof}

Combining all the steps, we have obtained an equivalence of triangulated categories:
\be
\CF: D^b U(\tgv^\hbar)\Mod^{gr}_G\simeq \Perf_{\C^\times}([\CM]\lbb\hbar\rbb). 
\ee
Note that the functor $\CF$ restricted to $ D^b U(\tgv^\hbar)\Mod^{gr}_G$ can be defined without mentioning the categories $D^+$ and $D^-$, and is really a functor on the chain level (namely between DG categories of finite chain complexes). The use of $D^\pm$ manifests that $\CF$ as an equivalence by allowing an adjoint $\CG$, whose chain-level definition involves unbounded categories.

\subsubsection{Step 2: upgrading to a tensor functor}

In this section, we show that the functor $\CF$ upgrades to a tensor functor. Let us first of all treat both $D^b U(\tgv^\hbar)\Mod^{gr}_G$ and $\Perf_{\C^\times}([\CM]\lbb\hbar\rbb)$ as symmetric monoidal categories. We will denote the two tensor products by $\otimes_{U}$ and $\otimes_{\CM}$ respectively. 

\begin{Lem}
There is a canonical isomorphism:
\be
\CF	(M\otimes_{U} N)\cong \CF(M)\otimes_{\CM}\CF(N).
\ee
This makes $\CF$ into a tensor functor between symmetric monoidal categories. 

\end{Lem}

\begin{proof}

The canonical isomorphism is given by:
\be
\CF(M\otimes_U N)=\mathrm{CE}^*(\tgv^+)\otimes (M\otimes_{\C\lbb\hbar\rbb} N)\cong \lp \mathrm{CE}^*(\tgv^+)\otimes M\rp \otimes_{ \mathrm{CE}^*(\tgv^+)\lbb\hbar\rbb} \lp \mathrm{CE}^*(\tgv^+)\otimes N\rp,
\ee
Note that the differential on the LHS beautifully coincides with differential on the right. It is tedious but rudimentary to show that this isomorphism respects the symmetric monoidal structure. 

\end{proof}

The Lie algebra $\tgv\lbb\hbar\rbb$ can be viewed as an object in $D^b U(\tgv^\hbar)\Mod^{gr}_G$, and as such it has the structure of a Lie algebra object (thanks to Jacobi identity). The proof of  Proposition \ref{Prop:TM} shows that $\CF(\tgv\lbb\hbar\rbb)$ is identified with the tangent Lie algebra of $[\CM]$, as a Lie algebra object in $\Perf_{\C^\times}([\CM]\lbb\hbar\rbb)$. Given an object $M\in D^b U(\tgv^\hbar)\Mod^{gr}_G$, it canonically has the structure of a $\tgv\lbb\hbar\rbb$ module in the symmetric monoidal category $D^b U(\tgv^\hbar)\Mod^{gr}_G$. We have also seen that $\CF (M)$ is canonically a module of $\CF(\tgv\lbb\hbar\rbb)=\CT_{[\CM]}[-1]\lbb\hbar\rbb$. We claim:

\begin{Prop}\label{Prop:CFaction}
The functor $\CF$ canonically gives rise to a commutative diagram:
\be
\btik
D^b U(\tgv^\hbar)\Mod^{gr}_G \rar{\CF}\dar & \Perf_{\C^\times}([\CM]\lbb\hbar\rbb)\dar\\
\tgv\lbb\hbar\rbb\Mod \lp D^b U(\tgv^\hbar)\Mod^{gr}_G\rp\rar{\CF} &\CT_{[\CM]}[-1]\lbb\hbar\rbb\Mod\lp  \Perf_{\C^\times}([\CM]\lbb\hbar\rbb)\rp. 
\etik
\ee

\end{Prop}

\begin{proof}

Let $a_M: \tgv\lbb\hbar\rbb\otimes_{U} M\to M$ be the action morphism for $M\in D^b U(\tgv^\hbar)\Mod^{gr}_G$, and let $\alpha_N: \CF(\tgv\lbb\hbar\rbb)\otimes_{[\CM]} N\to N$ be the action morphism for $N\in  \Perf_{\C^\times}([\CM]\lbb\hbar\rbb)$. The above commutativity amounts to $\CF(a_M)=\alpha_{\CF(M)}$. We recall that there is a short exact sequence of Lie algebras:
\be
\btik
0\rar & \tgv^+\lbb\hbar\rbb\rar & \tgv\lbb\hbar\rbb\rar & \fg\lbb\hbar\rbb\rar & 0
\etik
\ee
which incudes an exact triangle in $\Perf_{\C^\times}([\CM]\lbb\hbar\rbb)$ after applying $\CF$. This short exact sequence is exactly the short exact sequence at the end of Section \ref{subsubsec:tangentfiber}.  As we have stated in Lemma \ref{Lem:SESmodule}, given a module $N$ of $\CF( \tgv\lbb\hbar\rbb)$, we obtain a module of $\CF( \tgv^+\lbb\hbar\rbb)$ by restriction. Moreover, the derived invariants $(N)^{\CF( \tgv^+\lbb\hbar\rbb)}$ is a module of $\CF( \fg\lbb\hbar\rbb)$. The functor of $\CF( \tgv^+\lbb\hbar\rbb)$ derived invariants coincides with pull-push functor along the natural map $p: \widehat{\mu^{-1}(0)\times \mu^{-1}(0)}\to \mu^{-1}(0)$, and the action of $\CF( \fg\lbb\hbar\rbb)$ remaining is simply the derivative of the $G$-equivariant structure along the fibers of $p$. Finally, the action $\alpha_N$ is determined uniquely by its restriction to $\CF( \tgv^+\lbb\hbar\rbb)$, together with the action of $\CF( \fg\lbb\hbar\rbb)$ on $(N)^{\CF( \tgv^+\lbb\hbar\rbb)}$. Therefore we only need to prove the following two points. 

\begin{enumerate}

\item That $\CF(a_M)=\alpha_{\CF(M)}$ after restricting to $\CF(\tgv^+\lbb\hbar\rbb)$. 

\item That on the object $\CF(M)^{\CF(\tgv^+\lbb\hbar\rbb)}\cong \CF(M^{\tgv^+\lbb\hbar\rbb})$, the action of $\CF(\fg\lbb\hbar\rbb)$ coincides. 

\end{enumerate}

To prove point 1, notice that the Lie subalgebra $\CF(\tgv^+\lbb\hbar\rbb)$ is the tangent Lie algebra of $\mu^{-1}(0)$, viewed as a $G$-equivariant object, and the map $\CF(\tgv^+\lbb\hbar\rbb)\to \CF(\tgv\lbb\hbar\rbb)$ is the natural push-forward of tangent Lie algebras along the map $\pi: \mu^{-1}(0)\to [\CM]$. Since pullback is conservative, we only need to show $\pi^*\CF(a_M)=\pi^*\alpha_{\CF(M)}$ restricted to $\pi^*\CF(\tgv^+\lbb\hbar\rbb)$. This equality now follows from the definition of the action morphism of $\CF(\tgv^+\lbb\hbar\rbb)$ on $\CF(M)$ in Theorem 2.3.1 of \cite{hennion2018tangent}, as $\C[\mu^{-1}(0)]\cong \mathrm{CE}^*(\tgv^+)$. 

To prove point 2, consider the natural quotient map $\pi: [\CM]\to [\mathrm{pt}/G]$. The following is clearly a commutative diagram:
\be
\btik
D^b U(\tgv^\hbar)\Mod^{gr}_G \rar{\CF} \dar{(-)^{\tgv^+\lbb\hbar\rbb}}& \Perf_{\C^\times}([\CM]\lbb\hbar\rbb)\dar{\pi_*}\\
D(U(\fg\lbb\hbar\rbb))\Mod^{gr}_G\rar& \QCoh_{\C^\times}\lp [\mathrm{pt}/G]\lbb\hbar\rbb\rp
\etik
\ee
Under the identification $\widehat{\mu^{-1}(0)\times \mu^{-1}(0)}\cong \mathrm{Spec}(\mathrm{CE}^*(\CF(\tgv^+)))$, the map 
\be
p: \widehat{\mu^{-1}(0)\times \mu^{-1}(0)}\to \mu^{-1}(0)
\ee
is identified with the natural map of rings $\C[\mu^{-1}(0)]\to \mathrm{CE}^*(\CF(\tgv^+))$, such that the action of $\mu^{-1}(0)\times \wh{G}$ is identified with the action of $G$ on $\mathrm{CE}^*(\tgv^+)$. Therefore, to finish the proof, we are left to show that the canonical action of $T_{[\mathrm{pt}/G]}[-1]\cong \fg$ on $\pi_*(\CF(M))$ coincides with the action of $\fg$ on $M^{\tgv^+\lbb\hbar\rbb}$. This is now clear by definition since both are defined by the derivative of the action of $G$ on $\pi_*(\CF(M))\cong M^{\tgv^+\lbb\hbar\rbb}$. The proof is now complete. 

 \end{proof}

\subsubsection{Step 3. concluding the equivalence}

What we have shown so far is that we have an equivalence of symmetric monoidal categories:
\be
\CF: D^b U(\tgv^\hbar)\Mod^{gr}_G\longrightarrow \Perf_{\C^\times}([\CM]\lbb\hbar\rbb)
\ee
that intertwines the canonical action of $\tgv^\hbar$ on the LHS with the action of $\CF(\tgv^\hbar)$ on the RHS via the action of tangent Lie algebra. Moreover, according to Proposition \ref{Prop:TM}, the Poisson bi-vector on the RHS is naturally the image of the Quadratic Casimir on the LHS. Since this element in the symmetric tensor, together with its corresponding KZ equation is the only ingredients used to define the braided tensor category structure on both sides, by uniqueness of solutions of KZ equation in the case of nilpotent action of the quadratic Casimir \cite[Theorem A']{drinfeld1991quasi}, we obtain that $\CF$ can be upgraded to an equivalence of braided tensor categories. We have proven Theorem \ref{Thm:tangentBTC}.

\begin{Rem}

It is fairly easy to see from the construction of $\CF$ that it is in fact an equivalence of ribbon categories.

\end{Rem}

Combining Theorem \ref{Thm:tangentBTC} together with Theorem \ref{Thm:KLequiv}, we find the following corollary.

\begin{Cor}
There is an equivalence of triangulated ribbon categories:
\be
D^b\CA^\hbar\Mod^{gr}_G\longrightarrow \Perf_{\C^\times}([\CM]\lbb\hbar\rbb).
\ee

\end{Cor}

\begin{Rem}
This in particular answers a question of \cite{roberts2010rozansky} in the case of the stack $[\CM]$, that there exists a quantum group controlling the category of sheaves on $[\CM]$, which is explicitly a deformation of the universal enveloping algebra of the tangent Lie algebra as a Hopf algebra. It is an interesting question whether one can make $\CA^\hbar$ into a sheaf of Hopf algebras over $\CM$, which seem to not have been considered in literature. 

\end{Rem}

\newpage

  \bibliographystyle{alpha}
   
   
   \bibliography{QS}

\end{document}